\documentclass[12pt]{amsart}

\usepackage{geometry}
\usepackage{amsmath}
\usepackage{amssymb,amsfonts,latexsym}

\usepackage{verbatim}
\usepackage[all]{xy}

\usepackage{mathrsfs}
\usepackage[square, numbers, comma]{natbib}

\usepackage{amsbsy}

\usepackage{pb-diagram,pb-xy}
   \bibpunct{[}{]}{,}{n}{}{,}

\newtheorem{thm}{Theorem}[section]
\newtheorem{cor}[thm]{Corollary}
\newtheorem{lem}[thm]{Lemma}
\newtheorem{prop}[thm]{Proposition}
\theoremstyle{definition}

\theoremstyle{remark}
\newtheorem{rem}[thm]{Remark}

\newtheorem{exam}[thm]{Example}
\numberwithin{equation}{section}

\long\def\forget#1\forgotten{}

\newcommand{\eps}{\varepsilon}

\newcommand{\fp}{\mathfrak{p}}
\newcommand{\fP}{\mathfrak{P}}
\newcommand{\fq}{\mathfrak{q}}

\newcommand{\mQ}{\mathbb{Q}}

\newcommand{\mS}{\mathcal{S}}
\newcommand{\mP}{\mathbb{P}}

\newcommand{\Z}{\mathbb{Z}}

\DeclareMathOperator\charak{char}

\DeclareMathOperator\ld{ld}
\DeclareMathOperator\ed{ed}

\DeclareMathOperator\pd{pd}
\DeclareMathOperator\hgd{hgd}

\newcommand{\ra}{\rightarrow}

\DeclareMathOperator\Gal{Gal}
\def\({\left(}
\def\){\right)}
\newcommand\divides{\mid}

\newcommand\ndivides{\nmid}

\newcommand\oline[1] {{\overline{#1}}}

\newcommand\Hom{{\operatorname{Hom}}}

\newcommand\Aut{{\operatorname{Aut}}}

\renewcommand\Im{{\operatorname{Im}}}
\newcommand\HLG{{\operatorname{H}}}
\newcommand\ZLG{{\operatorname{Z}}}

\usepackage[colorlinks,pdftex, bookmarks=false]{hyperref}

\begin{document}

\title[The local dimension of a finite group] 
{The local dimension of a finite group over a number field} 

\author{Joachim K\"onig}
\address{Department of Mathematics Education, Korea National University of Education, Cheongju, South Korea}

\def\technion{Department of Mathematics, Technion - IIT, Haifa, Israel}
\def\ann{Department of Mathematics, University of Michigan, Ann Arbor}
\author{Danny Neftin}
\address{\technion}
\begin{abstract} 
Let $G$ be a finite group and $K$ a number field. We construct a $G$-extension $E/F$, with $F$ of transcendence degree $2$ over $K$, that specializes to all $G$-extensions of $K_\fp$, where $\fp$ runs over all but finitely many primes of $K$. 
If furthermore $G$ has a generic extension over $K$, we show that the extension $E/F$ has the so-called Hilbert--Grunwald property. These results are  compared to the notion of essential dimension of $G$ over $K$, and its arithmetic analogue.  
\end{abstract}

\maketitle

\section{Introduction}\label{sec:intro}
The essential dimension $\ed_K(G)$ of a finite group $G$ is a central notion in algebra which  measures the complexity of all $G$-extensions of overfields of a given field $K$ \cite{BR}. 
More precisely, $\ed_K(G)$ is the minimal integer $d\geq 0$ for which there exists a $G$-{\it extension} $E/F$, of a field $F$ of transcendence degree $d$ over $K$, such that every $G$-extension of $L$, for every field $L\supseteq K$, is a specialization of $E/F$ (over $K$). 
Here, a {$G$-extension} $E/F$ is an extension of \'etale algebras with Galois group $G$, and specialization is induced by taking residue extensions, as explained in \S\ref{sec:notation}. 

The {\it (essential) parametric dimension}
$\pd_K(G)$   is an analogous notion in arithmetic geometry which measures the complexity of all $G$-extensions of $K$ itself. More precisely, 
$\pd_K(G)$  is the minimal integer $d\geq 0$ for which there exist finitely\footnote{The motivation behind allowing finitely many $G$-extensions as opposed to one comes from rational connectedness, see Example \ref{exam:finite} and Remark \ref{rem:finite}.}  many $G$-extensions $E_i/F_i$, $i=1,\dots, r$, with $F_i$ of transcendence degree $d$ over $K$, such that every $G$-extension of $K$ is a specialization of $E_i/F_i$ for some $i$. 
In general, $\pd_K(G)\leq \ed_K(G)$, and  $\pd_K(G) < \ed_K(G)$ is possible even over the rationals $K=\mathbb Q$, see {Appendix} \ref{sec:parametric}. Over number fields $K$, the inequality  $\pd_K(G)>1$ is known for various classes of groups $G$, see \cite{Deb,KL,KLN, Koe19},  but in general  $\pd_K(G)$ remains a mysterious invariant. 

In the case where $K$ is a number field, the natural arithmetic approach to parametrizing $G$-extensions  $L/K$ is to first parametrize their completions $L_\fp:=L\otimes_K K_\fp$, where $\fp$ runs over primes of $K$. It is conjectured\footnote{This is a special case of the more general conjecture in \cite{CT} relevant to $G$-extensions, cf.\ \cite{Har,DLAN,LA}.} \cite{CT} that for every finitely many $G$-extensions $L^{(\fp)}/K_\fp$, $\fp\in S$, there exists a $G$-extension $L/K$  
such that $L\otimes_K K_\fp\cong L^{(\fp)}$ for all $\fp\in S\setminus T$, where $T=T(G,K)$ is a finite set of ``bad primes" depending only on $G$ and $K$. In this setting, the set of $G$-extensions $L^{(\fp)}/K_\fp, \fp\in S$ is  called a {\it Grunwald problem}, $L$ is called a {\it solution} to it, and the conjecture is known for many solvable groups $G$, see \cite{HW} and \cite[Theorem 9.5.9]{NSW}. The conjecture implies that finitely many $G$-extensions $E_i/F_i$, $i=1,\ldots,r$ which specialize to every $G$-extension of $K$, also specialize to every $G$-extension of $K_\fp$, for all but finitely many primes $\fp$ of $K$. This motivates defining 
the {\it (essential) local dimension} $\ld_K(G)$ of $G$ over $K$ as the minimal integer $d\geq 0$ for which there exist finitely many $G$-extensions $E_i/F_i$, $i=1,\ldots,r$, with $F_i$ of transcendence degree $d$ over $K$, such that every $G$-extension of $K_\fp$ is a specialization of some $E_i/F_i$, for all but finitely many primes $\fp$ of $K$. 

Grunwald problems are a topic of independent interest inspired by the inverse Galois problem (IGP) and the crossed product construction \cite[\S 11]{ABGV}. To parametrize solutions to Grunwald problems, we say that extensions $E_i/F_i$, $i=1,\ldots,r$ have the {\it Hilbert-Grunwald property} if
 every Grunwald problem $L^{(\fp)}/K_\fp, \fp\in S$ for $G$, with $S$ disjoint from a finite set of primes $T=T(E_1/F_1,\ldots,E_r/F_r)$, has a solution within the set of specializations 
of $E_i/F_i$, $i=1,\ldots,r$. 
In this context, the quantity of interest is the {\it Hilbert--Grunwald dimension} $\hgd_K(G)$ of $G$ over $K$, that is,  the minimal integer $d\geq 0$ for which there exist finitely many $G$-extensions $E_i/F_i$, $i=1,\dots, r$, with $F_i$ of transcendence degree $d$ over $K$, having the Hilbert--Grunwald property. 
Clearly, $\hgd_K(G)\geq \ld_K(G)$, and equality is possible in general, see Remark~\ref{rem:IGP}. Furthermore $\pd_K(G)\ge \hgd_K(G)$, conditional on the aforementioned conjecture.

The Hilbert--Grunwald property was first studied  by 
D\`ebes--Ghazi \cite{DG12} for a single $K$-regular $G$-extension $E/K(t)$ and Grunwald problems consisting of unramified extensions. 
Considering ramified extensions, \cite{KLN}  shows that $\ld_K(G)>1$ and hence $\hgd_K(G)>1$ for every group $G$ containing a noncyclic abelian subgroup \cite{KLN}, cf.\ Appendix \ref{sec:1-dim}. 
The notion of Hilbert--Grunwald dimension was first alluded to
 in \cite[\S 4]{Koe19}, which gave examples of (elementary abelian) groups with $\hgd_{K}(G)\leq 2$ and arbitrarily large essential dimension, giving rise to the question whether or not this invariant could ever be larger than $2$.

We show that $\ld_K(G)\leq 2$ holds in general, and  $\hgd_K(G)\leq 2$ holds if $G$ has a generic extension over $K$, or equivalently {\cite[Thm.\ 5.2.5]{Led}} 
if there is a $G$-extension of a rational function field over $K$ specializing to every $G$-extension of every overfield~of~$K$.  
\vskip 0.3cm

\noindent{\bf Main Theorem. }{\it 
Let $K$ be a number field and $G$ be a finite group. {Then:}\\
(1) {$\ld_K(G)\leq 2$. More precisely,} there exists a {single} $G$-extension $E/F$, with $F$ of transcendence degree $2$ over $K$, such that every $G$-extension of $K_\fp$ is a specialization of $E/F$, for every prime $\fp$ of $K$ outside a finite set $T$. \\
(2) If $G$ admits a generic extension over $K$, then {$\hgd_K(G)\le 2$. More precisely,} 
(a) $F$ {in (1)} can be chosen purely transcendental over $K$, and (b) every Grunwald problem $L^{(\fp)}/K_\fp$, $\fp\in S\setminus T$ has a solution within the set of specializations of $E/F$. }

\vskip 0.5cm

We note that although the assumption of {suitable} local global principles would imply $\pd_K(G)=\ld_K(G)=2$, see~Remark \ref{rem:LGP}, we expect $\pd_K(G)$ to be arbitrarily large as $G$ varies over finite groups. 
The main theorem and this expectation give a quantitative meaning to the assertion that the ``local complexity" of the Galois theory of a given finite group is in general much less than the corresponding ``global complexity". However,  little is known about groups $G$ with $\pd_K(G)>2$. 

Previous work on parametric extensions mainly restricted to the case where the fields $F_i$, {in the definition of $\pd_K(G)$}, are purely transcendental extensions of $K$. {It is therefore} also natural to consider the notion of ``rational parametric dimension", defined in the same way but with the additional assumption that the fields $F_i$ are rational, cf.\ \cite{DKLN}.  
{In fact, the rational analogue of essential dimension is well-known as the {\it generic dimension} (see, e.g., \cite{Sal}, \cite[\S8.5]{Led}). Note that when varying over groups  $G$ {that admit} a generic extension over $K$,  the generic dimension can be arbitrarily large, while it is unknown if the rational parametric dimension can be larger than $2$.}

The heart of the proof of the main theorem,  given in Section \ref{sec:proof1},  constructs  extensions of transcendence degree $2$ over $K$ that specialize to all tamely ramified local extensions {of $K_\fp$} with prescribed Galois group $D$ and inertia group $I$ {(for all but finitely many primes $\fp$ of $K$)}. 
Previous works \cite{Beck,KLN} give necessary conditions on such extensions depending on the local behavior at branch points. Here we give sufficient conditions on such extensions to maintain this local behavior under specialization, cf.\ Proposition \ref{prop:induct}. 
These conditions are achieved by the construction of Section \ref{sec:proof1}. Namely, we construct a transcendence degree $1$ extension that specializes to all unramified local extensions with group $D/I$ based on specialization methods of \cite{DG12}. We embed this extension into an extension of transcendence degree $2$, whose completion has Galois group $D$, inertia group $I$, and satisfies the above sufficient conditions, see Corollary \ref{thm:construct-2}. 

The constructed extensions $E/F$ in fact have the stronger property of  specializing to all $G$-extensions of $K'_{\fp'}$ for all but finitely many primes $\fp'$ of $K'$, for every finite extension $K'/K$. Thus, they relate to the notion of ``arithmetic dimension" introduced by  O'Neil, Bernstein, and Mazur, see Section \ref{sec:ext}.

The assumed existence of a generic extension for $G$ over $K$ in part (2) is  equivalent to the retract rationality of the variety $X:=\mathbb A^n/G$, for  $G\leq S_n$ \cite[Theorem 5.2.3]{Led}. 
This property is known to hold for many groups, cf.\ \cite{Led}, and most classically for $G=S_n$. 
In fact, in the proof of the main theorem, this assumption on $X$ can be relaxed to a (two dimensional) weak approximation property, see Remark \ref{rem:WWA}. Moreover, upon  replacing the single $G$-extension $E/F$ in the main theorem with finitely many $G$-extensions $E_i/F_i$, $i=1,\ldots,r$, we expect that this weak approximation property would hold for many more groups $G$, see Remark \ref{rem:prediction}. Adding finitely many extensions  also allows specializing to the $G$-extensions of $K_\fp$, {at the remaining primes $\fp\in T$}, see Section \ref{sec:finitely}.

Finally, we note that the notion of local dimension can be defined analogously for others fields $K$ equipped with (infinite) sets of valuations. The specialization methods developed in {Section \ref{sec:specialize}, and the more basic Section \ref{sec:dec},} are expected to be applicable over various other fields $K$.  It would be interesting to see how invariants such as $\pd_K(G)$ or $\ld_K(G)$ change as $K$ varies, cf.\ Remark \ref{rem:ret}. 

\vskip 5mm
We are grateful to the referee for helpful comments. The first and second author were supported by the National Research Foundation of Korea (grant no.\ 2019\ R1C1C1002665) and the Israel Science Foundation (grant no.\ 577/15), respectively. 

\section{Preparations}\label{sec:prep}
\subsection{Notation and setup}\label{sec:notation}
Let $K$ be a field of characteristic $0$. Let $\oline K$ denote the algebraic closure of $K${;}  $\mu_e$ the $e$-th roots of unity in $\oline K$; and $G_K:=\Gal(\oline K/K)$ the absolute Galois group of $K$. 
Although Section \ref{sec:intro} refers to $G$-extensions of \'etale algebras, throughout the rest of the paper, unless  explicitly mentioned, a $G$-extension refers to a $G$-extension of fields, cf.\ Section \ref{sec:etale} for the relation between the two.
We write $I\lhd R$ to denote that $I$ is an ideal of a ring $R$. In case  $K$ is the fraction field of a Dedekind domain $R$,  and $\fp\lhd R$ is a prime ideal of $R$, 
denote by $K_\fp$ the completion of $K$ at $\fp$.

\subsubsection{Morphisms and function fields}
{Let $F/K$ be a function field in finitely many variables with (exact) {\it constant field} $K$, i.e., such that $F\cap \overline{K} = K$. 
Let} $\varphi: G_{F}\to G$ be an epimorphism. The fixed field $E$ of $\ker\varphi$ is then Galois over $F$ with Galois group $G$. 
We may then choose a morphism $f:X\to Y$ of (smooth quasiprojective irreducible) varieties over $K$ such that {$E/F$ is isomorphic to the extension $K(X)/K(Y)$ of function fields of $X$ and $Y$, respectively}.
Note that 
$Y$ is absolutely irreducible, whereas $X$ is irreducible over $K$, but not necessarily absolutely irreducible.
For short, call such $f$ a {\it Galois cover} with group $G$. 
The epimorphism $\varphi$ factors through an epimorphism $\varphi': \pi_1(Y')_K\to G$ from the $K$-fundamental group of $Y':=Y\setminus D$ (where $D\subset Y$ is the branch divisor of $f$). In this setting, one says that $\varphi$ (or $\varphi'$) {\it represents}~$f$.

The {\it morphism of constants} 
$\varphi_K$ of $\varphi$ (resp.\ $\varphi'$) is the natural projection $\varphi_K: G_K\ra  G_K/N$, where $N$  is the image of $\ker\varphi$ under the natural projection $G_F\ra G_K$ (resp.\ $\pi_1(Y')_K\ra G_K$). The fixed field of $\ker\varphi_K$ is then the {constant field} of the fixed field $E$ of $\ker\varphi$. 
We shall say that $\varphi$ (or $E/F$, or $f$) is {\it $K$-regular}  if $\ker\varphi_K = G_K$ (or equivalently if $\varphi(G_{\oline KF})=G$, or if $\varphi( \pi_1(Y')_{\overline{K}})=G$).  
This ensures that $\ker\varphi\cdot  \pi_1(Y')_{\overline{K}} =  \pi_1(Y')_K$, and hence $E=\oline F^{\ker\varphi}$ is linearly disjoint from $\oline K$.

{Finally, note that, via Galois correspondence, one also has a correspondence of degree-$d$ (not necessarily Galois) covers $f:X\to Y$ and homomorphisms $\varphi: \pi_1(Y')_K \to S_d$, 
 for every $d\in \mathbb{N}$.} Under this correspondence, the covers $f$ with $X$ irreducible correspond to those $\varphi$ with $\varphi(\pi_1(Y')_K)\le S_d$ transitive. Cf.\  \cite[\S 2]{DL}.

\subsubsection{Specialization}
Given a  Galois cover $f:X\to Y$ over $K$ and a $K$-rational point $t_0\in Y(K)$ away from the branch divisor of $f$, one has a section $s_{t_0}: G_K\to  \pi_1(Y')_K$ to the projection  $\pi_1(Y')_K\to G_K$,
 and thus a {\it specialization morphism}
 $\varphi_{t_0} =\varphi'\circ s_{t_0}: G_{K}\to G$ (well defined up to conjugation). Denote the function field extension corresponding to $f:X\to Y$ by $E/F$ and the fixed field  of $\ker(\varphi\circ s_{t_0})$ by $E_{t_0}$. The {\it specialization} $E_{t_0}/K$ of $E/F$ at $t_0$ is then the residue extension at some (and, as  $f$ is Galois, any) point of $X$ in $f^{-1}(t_0)$. 
 We  also refer to $E_{t_0}/K$ as the specialization 
 {at} the $K$-rational place on $F$ corresponding to $t_0\in Y(K)$. The specialization $E_{t_0}/K$ is Galois with  group a subgroup of $G$.

 In case $F$ is the fraction field of a Dedekind domain $R$, and the valuation ideal $\fp\lhd R$ of $\nu$ is a nontrivial ideal of $R$, the specialization $E_{t_0}$ coincides with the residue field $S/\fP$, where $S$ is the integral closure of $R$ in $E$, and $\fP$ is a prime of $S$ lying over $\fp$.
 The key case for us is when $F=K(t)$ is a rational function field, and $R=K[t]$. For any $t_0\in \oline K$, 
  denote by $t\mapsto t_0$ the place corresponding to the prime $(m_{t_0})\lhd K[t]$, where $m_{t_0}$ is the minimal polynomial of $t_0$ over $K$. 
The specialization $E(t_0)_{t_0}/K(t_0)$ of 
$E(t_0)/F(t_0)$ at $t\mapsto t_0$ in the above sense is then the same as the residue extension at (any prime ideal extending) $m_{t_0}$ in $E/F$. 

\subsubsection{\'Etale algebras}\label{sec:etale}
In Section \ref{sec:intro}, the notion of a $G$-extension refers to an \'etale algebra $E$  that is Galois over $F$ with  group $G$. 
Such an algebra is induced from an $H$-extension $E_1/F$ of fields for $H\leq G$, cf.\ \cite[\S 4.3]{Led} for the following and further details. In particular, it is a direct sum $E=E_1\oplus\cdots\oplus E_r$ of fields with $G$ permuting the fields {$E_1,\ldots,E_r$} transitively, so that 
these are mutually $F$-isomorphic{. The $K$-isomorphism class of the extensions $E_1,\ldots,E_r$ of $K$} is called the {\it field underlying} $E$. 
Note that $E_1$ is Galois over $K$ with group $H$, where $H$ is the stabilizer of $G$ in its action on $\{E_1,\ldots,E_r\}$. The stabilizer of $E_i$ is then $\sigma H\sigma^{-1}$, for $\sigma \in G$ with  $\sigma(E_1)=E_i$, 
{so that the Galois group over $K$ of the underlying field is identified with a subgroup of $G$ up to conjugation.} 

Letting $L/K$ denote the specialization of $E_1/F$ at a $K$-rational place of $F$, we deduce that $L/K$ is Galois with group a subgroup of $H$. The specialization of the \'etale algebra $E/F$, referred to in Section \ref{sec:intro}, is then the $G$-extension of \'etale algebras induced from $L/K$.

\subsection{ID pairs}\label{sec:id-pairs}
Let $K$ be the fraction field of a Dedekind domain $R$ of characteristic $0$. 
A group $D$ and a subgroup $I$ are called an {\it ID (inertia-decomposition) pair over $K$} if 
$D$ appears as {the} Galois group of a tamely ramified field extension $T_\fp/K_\fp$ with inertia group $I$ for some prime $\fp$ of $R$. 
 Note that as part of an ID pair, the embedding $I\ra D$ is also fixed. The pair is called split if the projection $D\ra D/I$ splits.  
For an ID pair $\pi=(I,D)$, let $P_K(\pi)$ denote the set of primes $\fp$ for which there exists a tamely ramified field extension $T_\fp/K_\fp$ with Galois group $D$ and inertia group $I$. It is well known  \cite[Theorem 16.1.1]{Efrat} that $I\lhd D$,  and that $D/I$ is the Galois group of the residue extension at $\fp$.

The following lemma summarizes basic properties of $(I,D)$ pairs. Let $C_D(I)$ denote the centralizer of $I$ in $D$.
Let  $\sigma_{q,e}\in\Aut(K(\mu_e))$ denote the automorphism given by $\sigma_{q,e}(\zeta)=\zeta^q$ for all $\zeta\in \mu_e$. If a prime $\fp$ of $K$ has finite residue field set $N(\fp):=|R/\fp|$ to be its norm, {otherwise set $N(\fp)=\infty$}.

\begin{lem}
\label{lem:id_pairs}
Let $\pi=(I, D)$ be an ID pair over a fraction field $K$ of a Dedekind domain $R$ of characteristic $0$. Then: 
\begin{enumerate}
\item There exists a (unique) homomorphism 
$\eta_\pi:D\ra \Gal(K(\mu_e)/K)$ such that   $I\cong \mu_e$ as $D$-modules, where $e=|I|$, and $I$ is a $D$-module via conjugation while $\mu_e$ is a $D$-module via $\eta_\pi$.  In particular, $\ker(\eta_\pi)=C_D(I)$.
\item Assuming further that $K$ is a number field, $D/I$ and $\Im(\eta_\pi)$ are cyclic. Moreover, a prime $\fp$ of $K$ with norm $q$ is in $P_K(\pi)$ if and only if  $\langle\sigma_{q,e}\rangle =\Im(\eta_\pi)$.  
\end{enumerate} 
\end{lem}
\begin{proof}
Let $\fp\lhd R$ be a prime ideal and  $T_{\fp}/K_\fp$  a tamely ramified extension with Galois group $D$ and inertia group $I$. 

\noindent (1)
Since $T_\fp/K_\fp$ is tamely ramified, \cite[Prop.~6.2.1 and Cor.~16.2.7.(c)]{Efrat} show {the following: firstly, the residue field of $T_\fp^I$ contains the $e$-th roots of unity, and thus (via Hensel lifting) $\mu_e\subseteq T_\fp^I$. Secondly,} there is an isomorphism $\iota:I\ra \mu_e$ of $\Gal(T_\fp^I/K_\fp)$-modules, where $I$ is a module over $\Gal(T_\fp^I/K_\fp)=D/I$ via conjugation in $D$. In particular, $I$ is cyclic. 
 Identifying $I$ and $\mu_e$ via $\iota$, we obtain a homomorphism $D\ra\Aut(I)=\Aut(\mu_e)$ via the conjugation action in $D$. Let $\eta_\pi:D\ra\Aut(K(\mu_e))$ be its composition with the natural inclusion $\Aut(\mu_e)\ra\Aut(K(\mu_e))$, and let $V_\pi:=\Im(\eta_\pi)$ be its image. One {then} has $V_\pi\leq \Gal(K(\mu_e)/K)$ since  $D/I=\Gal(T_\fp^I/K_\fp)$ fixes $K_\fp$ and hence $K$. 
Since $\eta_\pi$ sends every element $d\in D$, that satisfies $\tau^d = \tau^q$ for all $\tau\in I$, to $\sigma_{q,e}\in \Gal(K(\mu_e)/K)$, the map $\eta_\pi$ is independent of the choice of $\iota$. Hence $\eta_\pi$ is the unique map with the required property. 
Note that the kernel of the conjugation action on $I=\mu_e$ is  $C_D(I)$, so that $\ker(\eta_\pi) = C_D(I)$. 

\noindent (2) 
 Since  $R/\fp$  is finite,  the Galois group $D/I$ of the residue extension  is cyclic, generated by  the Frobenius element  $\sigma\in D/I$. As $I\leq C_D(I)$, we have $\Im(\eta_\pi)\cong D/C_D(I)$ is cyclic. 

\noindent ``Only if part": Without loss of generality we may take $\fp$ and  $T_\fp/K_\fp$ to be as above. Let $\sigma\in D$ be a lift  of the Frobenius element  $\sigma\in \Gal(T_\fp^I/K_\fp) = D/I$. Since $\langle\sigma\rangle=D/I$ and $\sigma$ acts on $\mu_e$ by raising to the power of $q=N(\fp)$, the element $\eta_\pi(\sigma)=\sigma_{q,e}$ generates  $V_\pi$. \\
\noindent ``If part": 
 Let $K_\fq^{tr}$ be the maximal tamely ramified extension of $K_\fq$. By Iwasawa's theorem \cite[Theorem 7.5.3]{NSW}, 
its Galois group $G_{\fq}^{tr} =\Gal(K_\fq^{tr}/K_\fq)$ is isomorphic to the semidirect product $\langle \tau_\fq\rangle \rtimes \langle \sigma_\fq\rangle$ where $\langle\tau_\fq\rangle$ is the (procyclic) inertia group, $\langle \sigma_\fq\rangle\cong \hat{\mathbb Z}$ and  $\tau_\fq^{\sigma_\fq}=\tau_\fq^q$. Since $\langle \sigma_{q,e}\rangle = V_\pi$, there exists $d\in \eta_\pi^{-1}(\sigma_{q,e})$ such that $\langle dI\rangle = D/I$. Therefore the epimorphism $\phi_\fq:G_{\fq}^{tr}\ra D$, given by $\sigma_\fq\ra d$ and $\tau_\fq\ra \tau$ for a generator $\tau$ of $I$, is well defined, is an {epimorphism} and maps the inertia group to $I$. 
\end{proof}

 \subsection{Embedding problems}\label{sec:EP}
We describe the relevant facts from \cite[III.5]{NSW}. 
A finite embedding problem over a field $K$ is a pair $(\varphi:G_K\ra G,\varepsilon:E\ra G)$, where $\varphi$ is a (continuous) epimorphism, 
and $\varepsilon$ is an epimorphism of finite groups. 
A (continuous) homomorphism $\psi:G_K\ra E$ is called a \textit{solution} to $(\varphi, \varepsilon)$ if the composition  $\varepsilon \circ \psi$ {is equal to} $\varphi$. 
A solution $\psi$ is called a \textit{proper solution} if it is surjective. Two solutions $\psi_1,\psi_2$ to $(\varphi,\varepsilon)$ are called equivalent if there exists $e\in E$ such that $\psi_1(\sigma)=e^{-1}\psi_2(\sigma)e$ for all $\sigma\in G_K$.  A solution $\psi$ to $(\varphi,\varepsilon)$ over $L\supseteq K$ is a solution to $(\varphi_{L},\varepsilon)$, where $\varphi_{L}:G_L\ra G$ is an extension of $\varphi_{|\Gal(\oline K\cdot L/L)}$ to $G_L$. 
In case $L/K(t)$ is $K$-regular and $K_1/K$ is a finite extension, $\psi:G_L\ra E$ is called $K_1$-regular if $K_1$ is the fixed field of {the kernel of the morphism of constants $\psi_K$ of $\psi$}, in the notation of \S\ref{sec:notation}. In this case, $K_1/K$ is the constant extension of $\oline L^{\ker\psi}/L$. 

If $(\varphi,\varepsilon)$ has a proper solution $\psi$ (over $K$), then by identifying $E$ (resp.~$G$) with $\Gal(M/K)$ (resp.~$\Gal(L/K)$), where $M$ and $L$ are the fixed fields of $\ker\psi$ and $\ker\varphi$, respectively,  the restriction map $\Gal(M/K)\ra\Gal(L/K)$ coincides with $\varepsilon$.

If $A:=\ker\varepsilon$ is abelian, then $G$ acts on $A$ via conjugation in $E$, and hence $G_K$ acts on $A$ through $\varphi$. If further $(\varphi,\varepsilon)$ has a solution $\psi$,  Hoechsmann's theorem implies 
that the set of equivalence classes of solutions to $(\varphi,\varepsilon)$ is in bijection with classes in $\HLG^1(G_K,A)$. More specifically, the bijection attaches to a class $[\chi]$ of  $\chi\in \ZLG^1(G_K,A)$ the solution $\chi\cdot \psi:G_K\ra A$, $\sigma\mapsto \chi(\sigma)\psi(\sigma)$. 

The embedding problem $(\varphi,\varepsilon)$ is called {\it Brauer} if the kernel $A$ is  isomorphic to $\mu_n$ as a $G_K$-module. In particular,  $A\cong \mu_n$ is fixed by $\ker\varphi$. A general reference on Brauer embedding problems is \cite[Chapter IV,\S7]{MM}.
By Kummer theory, $\HLG^1(G_K,\mu_n)\cong K^\times/(K^\times)^n$. More explicitely, letting $M$ and $L$ be as above, we have  $M=L(\sqrt[n]{a})$ for  $a\in L^\times$ and the field fixed by $\ker(\chi\cdot \psi)$ is  $L(\sqrt[n]{ab})$  
for  $b\in K^\times/(K^\times)^n$ corresponding to $[\chi]$.

\subsection{Relating arithmetic and geometric decomposition groups}\label{sec:dec}

In this section, $K$ is the fraction field of a Dedekind domain $R$ of characteristic zero such that its  residue field at every place is perfect. We describe the decomposition group of a specialization $E_{t_0}/K$ at a prime {$\fp$ of $K$}, when $t_0\in K$ 
{and a branch point $t_1$} of a $K$-regular $G$-extension $E/K(t)$ {meet at $\fp$}, giving a more explicit version of \cite{KLN}. {Here, two values $a,b\in \mathbb{P}^1(\overline{K}$) are said to meet at $\fp$, if there is a prime $\fP$ of $K(a,b)$ such that {either $v_\fP(a)$ and $v_\fP(b)$ are both nonnegative and $v_{\fP}(a-b)>0$, or they are both negative and $v_{\fP}(1/a-1/b)>0$)}.} {Henceforth when $t_1=\infty$, the notation $t-t_1$ stands for $1/t$, and $K(t_1)$ for $K$.}

Let $D_{t_1}$ (resp., $I_{t_1}$) be the decomposition group (resp., the inertia group) of $E(t_1)/K(t_1,t)$ at the prime $(t-t_1) \, K[t-t_1]$, and let $D_{t_1,\fp'}$ be the decomposition group of $E(t_1)_{t_1}/K(t_1)$ at a prime $\fp'$ of $K(t_1)$, so that $D_{t_1,\fp'}$ is canonically identified with a subgroup of $D_{t_1}/I_{t_1}$. 
Let $\varphi:D_{t_1} \rightarrow D_{t_1}/I_{t_1}$ be the natural projection.

For a prime $\fp$ of $R$, denote by $v_\fp:K\ra \mathbb Z\cup\{\infty\}$ the associated normalized valuation, and by $|\cdot |_\fp:K\ra \mathbb R$ an associated absolute value. 
If $\fp$ is not in a finite set of bad places $\mathcal S_1$ and  $t_1$ is $\fp$-integral (that is, integral over the localization $R_{(\fp)}$ of $R$ at $\fp$), {then given any $t_0\in K$ meeting $t_1$ at $\fp$}, \cite[Lemma 3.2]{KLN} associates to it a {\it unique} degree $1$ prime $\fp'=\fp'(t_0,t_1,\fp)$ extending $\fp$ to $K(t_1)$ such that $v_{\fp'}(t_0-t_1)>0$, that is, $t_0$ and $t_1$ meet at $\fp'$. 
For a separable monic polynomial $P(t,x)\in R[t][x]$, denote by $\Delta_x(t)\in R[t]$ the {\it  discriminant} of $P$ with respect to $x$. 
Let $\mathcal S_R(P)$ be the finite set of primes $\fp$ of $R$ where $\Delta_x(t)$ mod $\fp$ is either $0$ or the number of pairwise distinct roots of $\Delta_x$ in an algebraic closure decreases upon reduction mod $\fp$. Equivalently, $\mathcal{S}_R(P)$ is the set of primes $\fp$ which either divide the leading coefficient of $\Delta_x(t)$ or are such that the mod $\fp$ reduction of the {\it {radical}} of $\Delta_x(t)$ becomes inseparable. {Here, the radical of $\Delta_x(t)$ is a separable polynomial with the same {irreducible factors} (without counting multiplicities) as $\Delta_x(t)$.}

\begin{thm}
\label{thm:kln}
Let $P(t,x)$ be a separable monic polynomial over $R[t]$ with splitting field $E$ such that $E/K(t)$ is a $G$-extension. 
Denote by $\mathcal{S}_0$ ($=\mathcal{S}_0(E/K(t))$)
the union of the set of primes of $R$ dividing $|G|$ with the finite set $\mathcal{S}_R(P)$. 
Suppose $\fp$ is a prime of $R$ not in $\mathcal{S}_0$, $t_0\in \mathbb{P}^1(K)$ is not a branch point of $E/K(t)$ while $t_1$ is a finite branch point. 
Suppose 
$v_{\fp'}(t_0-t_1)$ is positive and coprime\footnote{This assumption is made merely to simplify the assertion, see \cite[Theorem 4.2]{KLN} for the general version. Note that the description of $\mathcal S_R(P)$ given in this section applies to this generalized version.} to $|I_{t_1}|$. 
Then:
\begin{itemize}
\item[(1)] The decomposition group $D_{t_0,\fp}$ is conjugate in $G$ to $\varphi^{-1}(D_{t_1, \fp'})$. 
\item[(2)] The unramified part of the completion $E_{t_0}\cdot K_\fp$ of $E_{t_0}$ at $\fp$ is $(E(t_1))_{t_1}\cdot K(t_1)_{\fp'}$. 
\end{itemize}
\end{thm}
 \begin{rem}
 \label{rem:exc_primes}
The same conclusion is shown in \cite[Theorem 4.1]{KLN} under the assumption that $\fp$ is not in a set $\mathcal{S}_{\rm{exc}}$ of exceptional places which is described as the union of the set $\mathcal{S}_{\rm{bad}}$ of bad primes from the Inertia specialization theorem \cite{Beck, Leg19}, and four other sets $\mathcal S_1, \mathcal S_2, \mathcal S_3, \mathcal S_4$. The set $\mathcal{S}_{\rm{bad}}$ consists of 
1) the set $\mathcal{S}_{\rm{bad},1}$ of primes $\fp$ of $K$ where two branch points $t_1,t_2\in\mathbb P^1(\oline K)$ meet; 
2) the set $\mathcal{S}_{\rm{bad},2}$ of primes $\fp$ with ``vertical ramification", i.e., such that $\fp R[t]$ is ramified (in the sense of \cite[Section 2]{Beck}) in the integral closure of $R[t]$ in $E$;
 3) the set $\mathcal{S}_{\rm{bad},3}$ of primes $\fp$ such that at least one branch point of $E/K(t)$ is not $\fp$-integral; and 4) the set $\mathcal{S}_{\rm{bad},4}$ of primes that divide the order of $G$. 
  
The sets $\mathcal{S}_i$, $i=1,\dots, 4$ are unions of sets $\mathcal{S}_i(t_1)$ where $t_1$ runs through the finite branch points of $E/F(t)$.
 The set $\mathcal{S}_1(t_1)$ consists
  of places  where the minimal polynomial of $t_1$ is either not $\fp$-integral or nonseparable. 
 The set $\mathcal S_3(t_1)$ is the set of places of $K$ ramified in $K(t_1)$. 
The set $\mathcal S_4(t_1)$ is then the set of primes $\fp$ at which $t_1$ meets another root $d$ of $\Delta_x$.\footnote{Note that since $t_1$ is a branch point, 
 $(m_{t_1}(x))$ divides the (relative) discriminant of $E/F(t)$, and hence the discriminant of $P$. Thus $t_1$ is a root of  $\Delta_x$.}
Finally, the set $\mathcal{S}_2(t_1)$ is the set of primes $\fp$ for which the following holds (cf.\ \cite[Lemma 4.3]{KLN}): there exists some intermediate field $M$ of $E \cdot E(t_1)_{t_1} / K(t)$ such that every primitive element $y$ of $M/K(t)$ has a Puiseux expansion in $s:=t-t_1$ in which some coefficient is of $\fp$-adic absolute value $|y|_\fp >1$ (equivalently, of negative $\fp$-adic valuation).
\end{rem}

\begin{proof}[Proof of Theorem \ref{thm:kln}]
In view of Remark \ref{rem:exc_primes}, 
this follows from \cite[Theorem 4.1]{KLN}, once we show that  for every $\fp\in \mathcal S_{exc}$ (as described above)  coprime to  $|G|$, either the number of distinct roots of  $\Delta_x$ or the $t$-degree of $\Delta_x$ decreases upon reduction mod $\fp$.
The former happens as soon as two roots of $\Delta_x$ meet at $\fp$ (in particular, as soon as two finite branch points meet at $\fp$), whereas the latter happens as soon as some root $t_1$ of $\Delta_x$ meets infinity at $\fp$ (since then the minimal polynomial of $t_1$ over $K$ cannot have $\fp$-integral coefficients, whence any multiple of it with coefficients in $R$, in particular $\Delta_x$, must have a leading coefficient divisible by $\fp$). This readily shows the claim for all $\fp \in \mathcal{S}_{\rm{bad},1} \cup \mathcal{S}_{\rm{bad},3} \cup \mathcal{S}_4$, see Remark \ref{rem:exc_primes}.
 
The set $\mathcal{S}_{\rm{bad},2}$ is contained in the set of those primes modulo which $\Delta_x(t)$ is congruent to $0$ (see, e.g., Addendum 1.4.c) of \cite{DG12}). We have {also} assumed $\fp$ is coprime to $|G|$, so that $\fp \notin \mathcal S_{\rm{bad},4}$.  

Next, fix a finite branch point $t_1$ and let $m_{t_1}$ be its minimal polynomial over $K$. Note that, since $m_{t_1}$ divides $\Delta_x$, the number of distinct roots (resp., the degree) of $\Delta_x$ decreases upon reduction mod $\fp$ as soon as $m_{t_1}$ is inseparable modulo $\fp$ (resp., has non-$\fp$-integral coefficients).
If $\fp\in \mathcal S_3(t_1)$, the prime $\fp$ ramifies in $K(t_1)/K$, 
and thus divides the discriminant of $m_{t_1}$, that is, $m_{t_1}$ is inseparable modulo $\fp$. If $\fp\in \mS_1(t_1)$, then $m_{t_1}$ is non-$\fp$-integral or inseparable mod $\fp$ by definition.

It thus remains to consider the set $\mathcal{S}_2(t_1)$. If the roots of the given polynomial $P(t,x)$ have  Puiseux expansions in $s:=t-t_1$ all of whose coefficients are of $\fp$-adic absolute value $\le 1$, then it is elementary to construct primitive elements for each intermediate field of $E\cdot E(t_1)_{t_1}/K(t)$ with the same property, e.g.\ by taking a suitable  polynomial combination of the roots of $P$ with $\fp$-integral coefficients. It therefore suffices to show that, if $\fp$ is coprime to $G$ and the number of distinct roots of $\Delta_x$  does not decrease mod $\fp$, then any root $z$ of $P(t,x)$ has a Puiseux series expansion $z=z(s)=\sum_{n\ge 0} a_n s^{n/e}$ in $s$ with coefficients $a_n$ of $\fp$-adic absolute value $\le 1$.\footnote{Note that the expansion of $z$ contains no negative $n$ due to integrality of $z$ by the choice of $P$. Also note that $|a_n|_\fp\leq 1$ means $|a_n|_\fP\leq 1$ for any prolongation {$\fP$} of $\fp$ to $K(a_n)$.} This holds essentially due to results by Dwork and Robba on convergence of $\fp$-adic power series  \cite{DwR}: By assumption {$t_1$} does not meet any other root $d$ of $\Delta_x(t)$ at $\fp$. Since additionally $\fp$ is chosen coprime to $|G|$, \cite[Proposition 4.1]{Zannier} (which extends slightly \cite[Theorem 2.1]{DwR}) implies that the power series $\sum_{n\ge 0} a_n y^n$ converges in the open disk of $\fp$-adic absolute value $1$. By \cite[Corollary 4.3]{Bilu} (and monicity of $P$), the claim $|a_n|_{\fp} \le 1$ follows for all $n$.
\end{proof}

Finally, when $P$ depends on a parameter $s$, 
we uniformly bound the mod $\fp$ residue classes of $s\mapsto s_0\in R_\fp$  for which $\fp \in \mathcal S_R(P(s_0,t,x))$. 
\begin{lem} \label{lem:goodprimes}
Let 
 $P\in R[s,t][x]$ a monic separable polynomial in $x$. 
Then there {exist a finite set $S=S_P$ of primes of $R$ and a constant $d=d_P$}, depending only on $P$, such that {for all primes $\fp$ of $R$ not contained in $S$,} the number of mod $\fp$ residue classes of values $s_0\in R_{\fp}$ for which $\fp\in \mathcal S_R(P(s_0,t,x))$ is at most $d$. 
\end{lem}
\begin{proof}
Let $\Delta_x\in R[t,s]$ be the discriminant of $P$ with respect to $x$. 
Let $a(s)\in R[s]$ denote the leading coefficient of $\Delta_x$ as a polynomial in $t$. 
{We may assume that $S$ contains all prime divisors of $a(s)$, and thus restrict to} $\fp\ndivides a(s)$. 
It follows that $a(s)$ has at most $a_P:=\deg_s a(s)$ roots in $R/\fp$. 
Restricting to  $s_0\in R_\fp$ with residue different from these $a_P$ residues, 
 the discriminant of $P(s_0,t,x)$ with respect to $x$ is $\Delta_x(s_0,t)\in R_\fp[t]$. By choice of $s_0$, its leading coefficient $a(s_0)$ is nonzero mod $\fp$, and hence the $t$-degree of $\Delta_x(s_0,t)\in R_\fp[t]$  remains  the same mod $\fp$. 

Consider the radical $\Delta_x'\in R[s][t]$ of $\Delta_x(s,t)$ as a polynomial in $t$. 
The discriminant $\Delta\in K(s)$ of $\Delta_x'$ with respect to $t$ is then a nonzero rational function in $s$. 
{We may assume that $S$ contains all primes $\fp$ where the reduction of $\Delta$ mod $\fp$ is either not defined or is zero}, and henceforth assume $\Delta$ mod $\fp$ is a nonzero rational function in $(R/\fp)(s)$. Letting $b_P$ be the sum of the degrees of  the numerator and denominator of $\Delta$, we have at most $b_P$ residue classes of roots and poles of $\Delta$ mod $\fp$. Away from these $b_P$ values $\Delta(s_0)\in R/\fp$ is nonzero, and hence the number of distinct roots of $\Delta_x'$ (and hence of $\Delta_x$) remains the same mod $\fp$. 
Thus the assertion holds {with the constant $d_P:=a_P+b_P$}. 
\end{proof}
\begin{rem} a)  {In the case where $R$ is the ring of integers of a number field (or more generally, a Dedekind domain with finite residue fields), the set $S$ may of course be dropped from the assertion of Lemma \ref{lem:goodprimes} , via enlarging the constant $d_P$ sufficiently, namely to at least the maximal norm of any prime in $S$.}
  \\
 b)   The 
{ring $R[s]$ (of regular functions on the affine line) in Lemma \ref{lem:goodprimes}} can be replaced by {the coordinate ring $R[V]$ of} an irreducible affine variety $V$ {over $R$} 
at the account of excluding finitely many primes as follows. Given a polynomial $P\in R[V][t,x]$,
 denote be the specialization of $P$ at $v\in R[V]$ by  $P^{(v)}\in R[t,x]$. The above proof then shows that: \\
{\it There exist a proper (closed) subvariety $V_P$ and a finite set of primes $\mathcal S_1$ of $R$ such that $\fp\not\in \mathcal S_{R_\fp}(P^{(v)})$ for every $v\in V(R_\fp)\setminus V_P(K_\fp)$. 
}
\end{rem}

Finally, we note that although the residue extension is fixed as $t_0$ varies through values as in Theorem \ref{thm:kln},  the ramified part is quite arbitrary; cf.\ \cite[Theorem 4.4]{KLN}.
\begin{thm}\label{thm:KLN2}
Let $E/K(t)$ be a $K$-regular Galois extension,
 and $t_1\in \oline K$ a branch point. 
Let $D, I$, and $N/K(t_1)$ be the Galois group, inertia group, and residue extension, respectively, of the completion at  $t\mapsto t_1$. Let $\fp\lhd R$ be a prime away from the finite set $\mathcal S_0$ of Theorem \ref{thm:kln}, 
with a degree $1$ prime $\fp'$ of $K(t_1)$ lying over it, and define $D'$ as $\varphi^{-1}(D_{t_1,\fp'})$, with the notation of Theorem  \ref{thm:kln}. 
Let $T_\fp/K_\fp$ be a $D'$-extension with inertia group $I$,   
such that $T_\fp^I \cong NK(t_1)_{\fp'}$ ($\cong NK_\fp$). 
Then there exist infinitely many specializations $E_{t_0}/K$ at $t_0\in K$ 
whose completion at $\fp$  is $T_\fp/K_\fp$. 
\end{thm}
\begin{rem}\label{rem:KLN2}
Let $P\in R[t,x]$ be a polynomial with splitting field $E$ {over $K(t)$, and} whose roots are integral over {$R[\tau]$, where $\tau:=t-t_1$}. One may write the completion of $E$ at $t\mapsto t_1$ as {$N((\sqrt[e]{a\tau}))$} for $e=|I|$ and $a\in N^\times$. 
 The exceptional set of primes in \cite[Theorem 4.4]{KLN} is instead given as the set $\mathcal{S}_{exc}'$ consisting of the exceptional primes of Theorem \ref{thm:kln} together with the finitely many primes $\fp$ for which  $e\ndivides v_{\fp'}(a)$ for some   degree $1$ prime $\fp'\divides\fp$  of the integral closure $R'$ of $R$ in $N$. 
However we claim that the latter primes are contained in $\mathcal S_R(P)$. 

First, by possibly replacing $a$, we may further assume that $a\in R'$ and hence $v_{\fp'}(a)>0$. As $e=|I|$ can be assumed to be $>1$, there is a root  $\alpha$ of $P$ that is not fixed by $I$, so that $\beta=\alpha^i$ is a nontrivial conjugate of $\alpha$ by some $i\in I$. Since $\alpha,\beta$ are integral over {$R[\tau]$}, they are integral over {$R'[[\tau]]$}. 
{Setting $u=\sqrt[e]{a\tau}$}, we note that since  $v_{\fp'}(a)>0$, the prime $\fp'$ divides  {$u^e=a\tau$}. {Writing $\alpha$ and $\beta$ as power series in $u$, the free coefficients of $\alpha$ and $\beta$ are the same and hence the difference $\alpha-\beta$ is a multiple of $u$. Thus $(\alpha-\beta)^e$ is a multiple of $u^e$ and hence divisible by $\fp'$ as an element in the integral closure of {$R'[[\tau]]$} in $N((u))$}. Letting $\Delta_x$ denote the discriminant of $P$, we deduce that $(\Delta_x)^e$ is zero mod $\fp'$, and hence $\Delta_x\equiv 0$ mod $\fp$, proving the claim.

\end{rem}


\subsection{Twists and specializations}\label{sec:twists}

Let $K$ be a field of characteristic $0$, let $f:X\to Y$ be a Galois cover with group $G$ of smooth projective irreducible varieties over $K$, with $Y$ absolutely irreducible. Let $\varphi: \pi_1(Y')_K\to G$ be an epimorphism representing $f$, as in Section \ref{sec:notation}. 
 We are interested in whether a given morphism $\psi: G_K\to G$ occurs as a specialization morphism of the given $\varphi$. This is equivalent to a question about $K$-rational points of a certain cover $f^\psi$, called the {\it twist} of $f$ by $\psi$. 
 
We follow the definition and properties of $f^\psi:\tilde{X}\to Y$ given in \cite[Section 3]{DL}\footnote{The twisting lemma can also be viewed as the well known field crossing argument.}
 Consider the representation 
 $\varphi^\psi: \pi_1(Y')_K\to Sym(G)$ 
  given by:
 $$(\varphi^\psi(\theta))(x) := \varphi(\theta) x (\psi(\overline{\theta}))^{-1},\  \theta \in \pi_1(Y')_K, x\in G,$$
where {$Sym(G)$ is the symmetric group on the set $G$ and} $\overline{\theta}$ is the image of $\theta$ under the natural projection $\pi_1(Y')_K \to G_K$.
Note that this action is not necessarily transitive. The equivalence of such representation{s} with \'etale covers yields $f^\psi$ as follows. Consider the orbits $O_1,\ldots,O_r\subseteq G$ in the above action, and let $H_i\leq \pi_1(Y')_K$ be the stabilizer of a point in $O_i$, $i=1,\ldots,r$. The quotient by $H_i$ defines an \'etale covering $X_i\ra Y'$, and $f^\psi$ is the resulting covering from $\tilde X=\prod_{i=1}^r X_i$ to $Y'$. 
 
The restriction of $\varphi^\psi$ to $\oline K$ is equivalent to that of $\varphi$ (composed with the regular representation), and hence the cover $f^\psi:\tilde{X}\to Y$ becomes isomorphic to $f$ over $\oline K$. 
 Moreover, by the so-called ``twisting lemma" \cite[Lemma 2.1]{DG12}, \cite[Lemma 3.2]{DL}, $\tilde{X}$ has a rational point at ${\bf t}_0\in Y(K)$ (away from the branch locus) if and only if the specialization morphism of $\varphi$ at ${\bf t}_0$ is {\it equivalent to $\psi$ by conjugation}, that is, equals the composition of  $\psi$ with an inner automorphism of $G$.
 
Note that a trivially necessary condition for $\tilde{X}$ to have a $K$-rational point is for $f^\psi$ to be $K$-regular, i.e., have trivial morphism of constants. This requires 
the morphism of constants $\varphi_K$ of $\varphi$ to fulfill $\varphi_K = \theta\circ \psi$, where (up to conjugation) $\theta$ is the canonical epimorphism of $G$ onto the image of $\varphi_K$. In this case, say for short that $\psi$ {\it projects} onto $\varphi_K$.

For a number field $K$ and a prime $\fp$ of $K$,  denote by $\varphi_\fp$ and $G_\fp$ the restriction of $\varphi$ to $\pi_1(Y')_{K_{\fp}}$ and its  image, respectively. 
We can now state the following theorem on specializations of function field extensions $E/K(t)$ with a prescribed local behavior, which extends \cite[Theorem 1.2]{DG12} to cases where $E$ is not necessarily $K$-regular. {Compare also to \cite[Corollary 4.4]{DL}, which has no regularity assumption, but considers only primes which split completely in the field of constants of the cover.}

\begin{thm}
\label{thm:debesghazi_ext}
Let $f:X\to \mathbb{P}^1$ be a Galois cover with group $G$, defined over {a number field} $K$, represented by
$\varphi: \pi_1((\mathbb{P}^1)')_K\to G$ and with function field extension $E/K(t)$.
Then there exists a bound $L(\varphi)$ satisfying the following property: for every prime $\fp$ of $K$ with norm $N(\fp)>L(\varphi)$ and every unramified morphism $\psi: G_{K_\fp} \to G$, with $\psi$ projecting onto the morphism of constants of $\varphi_\fp$, 
there exist infinitely many $t_0\in K$ such that the specialization of $\varphi_\fp$ 
at $t_0$ is equivalent to $\psi$ by conjugation.

In particular, for every unramified extension $F/{K_\fp}$ which contains the full constant field of $EK_\fp/K_\fp(t)$ and whose Galois group embeds into $G$, 
there exist infinitely many $t_0\in K$ such that the completion of $E_{t_0}/K$ at $\fp$ equals $F/K_\fp$.
\end{thm}
\begin{proof}
Via the twisting lemma, it suffices to show that the twist $f^\psi$, viewed as a morphism over $K_\fp$, 
has infinitely many $K_\fp$-points. As shown in \cite[Section 3.1.1]{DL}, the assumption that $\psi$ projects onto the morphism of constants of $\varphi_{\fp}$ implies that the restriction of the twist $f^\psi:X_\fp\to \mathbb{P}^1_{K_\fp}$, to a connected component $X_\fp$ over  $K_\fp$, is $K_\fp$-regular, that is, $X_\fp$ is absolutely irreducible. Furthermore, as stated above $X_\fp$ is of the same genus as a connected component of $X$ over $\oline K_\fp$, since the two become isomorphic after suitable base change. The same argument as in the proof of \cite[Theorem 1.2]{DG12} applies from here on: The first assertion follows via considering first the mod-$\fp$ reduction $\overline{X_\fp}$ of $X_\fp$, application of the Lang-Weil bound for this reduction in order to ensure the existence of simple mod-$\fp$ points, and then Hensel lifting to ensure the existence of $K_\fp$-points. The assertion about fields follows by choosing an epimorphism $G_{K_\fp}\to \Gal(F/K_\fp)$ which projects onto the prescribed morphism of constants of $\varphi_{\fp}$, and whose kernel has the constant field of $EK_\fp$ as a fixed field. This is possible since the constant field of  $EK_\fp $ is contained  in $F$.
\end{proof}

\begin{rem}
\label{rem:debesghazi_ext}
An explicit application of the Lang-Weil bound in the above proof shows the following stronger statement: There are constants $a,b>0$ depending only on $E/K(t)$ such that  for at least $a\cdot N(\fp)- b$ residue classes $r$ mod $\fp$, the condition $t_0 \equiv r$ mod $\fp$ is sufficient to yield the assertions of Theorem   \ref{thm:debesghazi_ext}. 
Indeed, this follows since at most  $|G|$ points of $\overline{X_\fp}$ lie over the same point of $\mathbb{P}^1$, and the only specialization values $t_0$ that need to be excluded {are} the bounded number of branch points.
\end{rem}

\section{The construction and its specializations}
\label{sec:proof1}
The construction of the main theorem is given in Section \ref{sec:construct} for a single ID pair $\pi=(I,D)$ over a number field $K$. Its specialization properties are given in Section~\ref{sec:specialize}.

\subsection{Construction}\label{sec:construct}
The following two lemmas are used to construct the residue extensions of the desired extension: 
\begin{lem}\label{lem:residue} Let $F_0$ be a one variable function field
{with exact constant} field $K$, let $L/K$ be a cyclic extension, and $\varphi:G_{F_0}\ra V$ the restriction map to $V:=\Gal(LF_0/F_0)$.  
Let $U$ be a cyclic group with an epimorphism $p:U\ra V$.
Then 
$(\varphi,p)$ has an $L$-regular proper solution over $M$ for some $K$-regular extension $M/F_0$. 
\end{lem}
\begin{proof}
Set $W:=\ker p$.
 Since $U$ is cyclic, {it is classical that, regardless of the base field $K$,} there exists a $K$-regular epimorphism $\varphi_2:G_{F_0}\ra U$ {(e.g., \cite[Proposition 16.3.5]{FJ} gives this for every abelian group over $K(t)$ instead of $F_0$; for the general case, in which $F_0$ is a finite extension of $K(t)$, simply apply the above for a sufficiently high direct power of the cyclic group $U$ to get at least one epimorphism $G_{K(t)}\ra U$ which remains surjective after restriction to $G_{F_0}$)}. 
Let $\varphi_1:G_{F_0}\ra V$ be the composition of $\varphi_2$ with $p$. 
Since $\varphi_1$ is $K$-regular, $\ker\varphi_1\cdot G_{\oline K F_0}=G_{F_0}$ and hence 
$\ker\varphi_1\cdot \ker\varphi = G_{F_0}$. 
Thus, the map $\varphi_1\times \varphi:G_{F_0} \ra V\times V$, $\sigma\mapsto (\varphi_1(\sigma), \varphi(\sigma))$, is onto. 
Choose $M$ to be the subfield fixed by $(\varphi_1\times \varphi)^{-1}(D_1)$, where $D_1\leq V\times V$ is the diagonal subgroup. 

Since the image of the map $\varphi_2\times \varphi:G_M\ra U\times V$ is the diagonal subgroup $D_2:=\{(u,p(u))\,|\,u\in U\}$, the projection $p_U:U\times V\ra U$ to the first coordinate  maps $D_2$ isomorphically to $U$, yielding an epimorphism $\psi:=p_U\circ (\varphi_2\times \varphi):G_M\ra U$. Letting $p_V:U\times V\ra V$ be the projection to $V$, since $p\circ p_U = p_V\circ (p\times id)$ on $D_2$, and $p\circ \varphi_2= \varphi_1$, we have $p\circ \psi = p_V\circ (p\times id)\circ (\varphi_2\times \varphi) =  \varphi.$
Thus $\psi$ is a (proper) solution. 
$$
\xymatrix{
& & G_M \ar[d]^{\varphi_1\times \varphi} \ar[dl]_{\varphi_2\times \varphi} \\
W\times 1\ar[r]^{id} \ar[d]_{p_U}^{\cong} & D_2 \ar[r]^{p\times id}\ar[d]_{p_U}^{\cong} & D_1 \ar[d]_{p_V}^{\cong} \\
W\ar[r]^{id} & U \ar[r]^{p} & V
}
$$

It remains to show that $\psi$ is $L$-regular. 
First note that $\psi^{-1}(W)=\ker\varphi_{|G_M}=G_{LM}$, and $LM/M$ is a constant extension. Thus to prove $L$-regularity, it suffices to show $\psi(G_{\oline KM})\supseteq W$.  We claim that $G_{\oline KM} \supseteq G_{\oline KF_0}\cap \varphi_2^{-1}(W)$. Since $\varphi_2(G_{\oline KF_0})=U$ as $\varphi_2$ is $K$-regular,  the claim shows that $\varphi_2(G_{\oline KM}) {\supseteq} W$, and hence $$\psi(G_{\oline KM})=p_U\circ (\varphi_2\times \varphi)(G_{\oline KM})\supseteq W,$$ as desired for $L$-regularity. 
To prove the claim, recall that $M$ is the fixed field of $(\varphi_2\times \varphi)^{-1}(D_2){=(\varphi_1\times\varphi)^{-1}(D_1)}$, which contains 
$(\varphi_2\times \varphi)^{-1}(W\times 1)=\varphi_2^{-1}(W)\cap \ker\varphi$. 
Thus,
$G_{\oline KM} =  G_{\oline KF_0}\cap (\varphi_2\times \varphi)^{-1}(D_2)
\supseteq G_{\oline K F_0}\cap \varphi_2^{-1}(W)\cap \ker\varphi.$
Since $\ker\varphi = G_{LF_0}\supseteq G_{\oline K F_0}$, 
this gives $G_{\oline K M}\supseteq G_{\oline KF_0}\cap \varphi_2^{-1}(W)$, proving the claim. 
\end{proof}
A stronger conclusion is available in case the extension $L/K$ is cyclotomic:
\begin{lem}\label{lem:p-power}
In the setup of Lemma \ref{lem:residue}, assume further that ${U}$ is {an $\ell$-group for a prime $\ell\neq \charak K$, where either  $\ell$ is on odd prime or $\mu_4\subseteq K$,  and that $L=K(\mu_e)$ for $e=\ell^s$}. Then $(\varphi,p)$ has an $L$-regular proper solution over $F_0$. 
\end{lem}
\begin{proof} 
Set $W:=\ker p$. 
{If $V$ is trivial, a solution is merely a $K$-regular realization of $U$, and hence follows as in the proof of Lemma \ref{lem:residue} from \cite[Proposition 16.3.5]{FJ}.}
{Henceforth assume $V$ is nontrivial, so that $\mu_\ell\subseteq K$}, and $s$ is the unique (maximal) integer such that {$L=K(\mu_{\ell^s})$}.
Since {$L=K(\mu_{\ell^s})$}, the embedding problem $(\varphi,p)$ has a solution $\psi_0:G_K\ra U$ given by restriction to {$\Gal(K(\mu_{\ell^{s+k}})/K)$}, where {$\ell^k=|W|$}. Extend $\psi_0$ to $G_{F_0}$ by composing with the restriction $G_{F_0}\ra G_K$. Let $\varphi_2:G_{F_0}\ra W$ be a $K$-regular epimorphism, that is, $\varphi_2(G_{\oline KF_0})={W}$. Since $V$ is abelian, {the product}  $\psi = \psi_0\cdot\varphi_2$  is also a solution to $(\varphi,p)$. Since $U$ is cyclic of {$\ell$-power} order {and $V\neq 1$}, a preimage under $p$ of a generator of $V$ is a generator of $U$, and hence $\psi$ is proper. 

It remains to show that $\psi$ is $L$-regular. Since $p\circ\psi = p\circ \psi_0$, the fixed field of $\ker(p\circ\psi)$ is $F_1:=F_0(\mu_e)$. Let $\psi_W:G_{F_1}\ra W$ (resp.~$\varphi_1:G_{F_1}\ra W$) be the restriction of $\psi$ (resp.~$\varphi_2)$ to $G_{F_1}$. 
Since $\varphi_2(G_{\oline KF_0})={W}$ as $\varphi_2$ is $K$-regular, 
we have $\varphi_1(G_{\oline KF_0})=W$. 
As $\psi_0(G_{\oline KF_0})=1$, we have $\psi(G_{\oline KF_0})=\varphi_1(G_{\oline KF_0})=W$, so that $\psi$ is $L$-regular. 
\end{proof}
We note that the only property of $L/K$ used in the proof is that it embeds into a $V$-extension over $K$. 

The following lemma is used to construct the desired $D$-extension from the residue $(D/I)$-extension:
\begin{lem}\label{lem:inertia}
Let $D$ be a finite group, $I\lhd D$ a cyclic normal subgroup of order $e$, and $p_2:D\ra D/I$ the natural projection. 
Suppose $F_\pi$ is a complete discrete valuation field with characteristic $0$ residue field $M$. 
Assume that $I\cong \mu_e$ as $D$-modules, where $D$ acts on $I$ by conjugation, and on $\mu_e$ via  $p_1\circ p_2$, where $p_1:D/I\ra \Gal(M(\mu_e)/M)$ is an epimorphism. 
Finally, {let} $\varphi_1:G_M\ra D$ be a solution to $(\varphi, p_1\circ p_2)$, where $\varphi$ is the restriction map.  Then  $(p_2\circ \varphi_1,p_2)$ has a proper totally ramified solution over~$F_\pi$. 
\end{lem}
$$
\xymatrix{
& & G_M \ar[d]^{\varphi} \ar@/_1.2pc/[dll]_{{\varphi_1}}\ar[dl]_{p_2\circ \varphi_1} \\ 
D \ar[r]^{p_2} & D/I \ar[r]^>>>{p_1} & \Gal(M(\mu_e)/M)
}
$$
\begin{proof}
Embedding the algebraic closure $\oline M$ of $M$ into that of $F_\pi$, we extend $\varphi$ and $\varphi_1$ to $G_{F_\pi}$. Since $\ker p_2=I\cong \mu_e$ as $G_{F_\pi}$-modules, the equivalence classes {of} solutions to $({p_2\circ \varphi_1},p_2)$ are in one to one correspondence with elements of $\HLG^1(G_{F_\pi},\mu_e)$, as described in Section \ref{sec:EP}. Moreover, letting $N$ be the fixed field of $\ker(p_2\circ\varphi_1)$, the fixed field of $\ker \varphi_1$ is of the form $N(\sqrt[e]{b})$ for  $b\in N^\times$. Pick $s_\pi\in F_\pi^\times$ to be a uniformizer, and let $\alpha_\pi\in \ZLG^1(G_{F_\pi},\mu_e)$ be a class corresponding to $s_\pi (F_\pi^\times)^e$  via the Kummer isomorphism $\HLG^1(G_{F_\pi},\mu_e)\cong F_\pi^\times/(F_\pi^\times)^e$. As in Section \ref{sec:EP}, the fixed field of  $\ker(\alpha_\pi\cdot \varphi_1)$ is $N(\sqrt[e]{bs_\pi})$. As $bs_\pi$ is also a uniformizer,  $\alpha_\pi\cdot \varphi_1$ is {a} proper totally ramified solution. 
\end{proof}
Combining the above lemmas we obtain the following construction.  Recall that Lemma \ref{lem:id_pairs} associates to every ID pair $(I,D)$ with $e:=|I|$, 
a unique map $\eta_\pi:D\ra   \Gal(K(\mu_e)/K)$ with kernel $C_D(I)$, whose image is denoted by $V_\pi$. 
\begin{cor}\label{thm:construct-2}
Let $\pi=(I,D)$ be an ID pair over a number field $K$, and let $\eta_\pi:D\ra V_\pi$ be as above. Let $t$ be transcendental over $K$, and $F_0=K(t)$. Then there exist infinitely many places $s\mapsto s_\pi$ of 
$F_0(s)$ such that the completion $F_\pi$ at $s\mapsto s_\pi$ admits a $D$-extension $E_\pi/F_\pi$ 
with inertia group $I$ and residue extension $N/M$ satisfying: 
\begin{enumerate}
\item[(a)] $M/K(t)$ is  finite  $K(\mu_e)^{V_\pi}$-regular; 
\item[(b)] the constant field of $N$ is $K(\mu_e)$;
\item[(c)] $N^{C} \cong M(\mu_e)$, where $C=C_D(I)$ acts via the natural quotient $C\ra C/I$. 
\end{enumerate}
\end{cor}
\begin{proof}
Let $U$ be a cyclic subgroup of $D$ which projects onto $D/I$ under the  projection $p_I:D\ra D/I$ modulo $I$. In particular, the action of $U$ on $I$ by conjugation factors through that of $D/I$. 
Set $K_0:=K(\mu_e)^{V_\pi}$, so that $K_0(\mu_e)/K_0$ is cyclic with Galois group $V_\pi$, and let $\varphi:G_{F_0}\ra V_\pi$ be the restriction map.
By Lemma \ref{lem:residue}, there exists a finite $K_0$-regular extension $M/F_0$ (so that (a) holds) over which $(\varphi,(\eta_\pi)_{|U}:U\ra V_\pi)$ has a proper $K_0(\mu_e)$-regular solution $\varphi_1:G_M\ra U$. 

Let $\hat D=I\rtimes U$ with semidirect product action given by conjugation in $D$, and let $\hat p_I:\hat D\ra U$ be the  projection modulo $I$. 
Let $s_\pi$ be a primitive element for $M/F_0$, and $F_\pi$ the completion of $F_0(s)$ at $s\ra s_\pi$. Since $(\varphi_1,{\hat p_I})$ has a trivial solution $\varphi_1$ {over $F_0(s)$ and hence over $F_\pi$}, it also has a proper totally ramified solution $\hat\psi:G_{F_\pi}\ra \hat D$ over  $F_\pi$, by Lemma \ref{lem:inertia}. Letting  $\hat p:\hat D\ra D$ be the natural projection, we note that $p_I\circ \hat p = p_I\circ \hat p_I$, and hence
$$p_I \circ \hat p\circ \hat\psi = p_I\circ \hat p_I \circ {\hat\psi}= p_I\circ \varphi_1,$$
so that $\psi:=\hat p\circ \hat\psi$ is a solution to $(p_I\circ \varphi_1:G_{F_\pi}\ra D/I,p_I:D\ra D/I)$. 

Since $\hat p:\hat D\ra D$ maps $I$ to itself, $\psi$ is a totally ramified {solution}. Thus, the extension  $E_\pi/F_\pi$ fixed by $\ker(\psi)$ is a $D$-extension 
with inertia group $I$.  In particular, its residue extension $N/M$  at $s\mapsto s_\pi$ is the extension fixed by the kernel of $\varphi_0=p_I\circ \varphi_1:G_M\ra D/I$. As $\eta_\pi$ has kernel $C\supseteq I$, the map $\eta_\pi$  factors as $\oline \eta_\pi\circ p_I$ for an epimorphism $\overline \eta_\pi:D/I\ra  V_\pi$. 
Thus, the kernel of $\oline\eta_\pi\circ \varphi_0 = \eta_\pi\circ \varphi_1=\varphi$ {has fixed field} $M(\mu_e)$. In particular, $N\supseteq M(\mu_e)$. 
Moreover, as $\ker\oline\eta_\pi = p_I(C)=C/I$, we have $N^{C}=M(\mu_e)$, giving (c). 
As $\varphi_1$ is $K(\mu_e)$-regular, so are $\varphi_0$ and $N/M$, giving~(b). 
\end{proof}

\begin{rem}\label{rem:p-power}
If $D$ is an $\ell$-group, where either $\ell$ is an odd prime or $\mu_4\subseteq K$,  
we can {furthermore} choose $M=K(\mu_e)^{V_\pi}(t)$. Indeed, the same proof applies when replacing the application of Lemma \ref{lem:residue} by  Lemma \ref{lem:p-power}. 
\end{rem}

\begin{cor}\label{cor:construct-1}Let $\pi=(I,D)$ be a split ID pair over a number field $K$ with $C_D(I)=I$, and  $\eta_\pi:D\ra \Gal(K(\mu_e)/K)$ the map from Lemma \ref{lem:id_pairs}.
Then there exist infinitely many places $t\mapsto t_\pi'$ of $K(t)$ such that the completion $F_\pi$ at $t\mapsto t_\pi'$ admits a $D$-extension  $E_\pi/F_\pi$ 
with inertia group $I$ and residue extension $K(\mu_e)/K(\mu_e)^{V_\pi}$. 
\end{cor}
\begin{proof}
Let $V_\pi=\Im(\eta_\pi)$ be the image of $\eta_\pi$. Letting $t_\pi'$ be a primitive element for $M:=K(\mu_e)^{V_\pi}$, and $F_\pi$ the completion {of $K(t)$} at $t\ra t_\pi'$. Consider the restriction map $\varphi: G_{M}\ra \Gal(M(\mu_e)/M)\cong D/I$  and the projection $p_2:D\ra D/I$ modulo $I$. 
The assertion then follows from Lemma \ref{lem:inertia} applied to the split embedding problem $(\varphi,p_2)$ with the trivial solution $\varphi_1=\varphi$. 
\end{proof}

\subsection{Specialization} \label{sec:specialize}
The following lemma {is used to pass} from a parametrization the residue $(D/I)$-extensions of $K_\fp$ to parametrizations of $D$-extensions of $K_\fp$. 

Let $M/K(t)$ be a finite Galois extension,  $C$ a finite group, and $\alpha,\beta\in \mathbb R$ be positive. We  say that an extension $N/M$ with constant field $L=N\cap \oline K$ has the property 
$P(C,\alpha,\beta)$ if for every prime $\fp$ of $K$ which is inert in $L$, and every unramified $C$-extension $U_\fp/K_\fp$, there are at least $\alpha N(\fp)-\beta$ (resp.~infinitely many if $N(\fp)=\infty$) residue classes $\oline t_0$ mod $\fp$, and for each class infinitely many degree $1$ places $t\mapsto t_0\in K_\fp$ with $t_0\equiv \oline t_0$ (mod $\fp)$  such that $U_\fp$ is the specialization of $N$ at $t\mapsto t_0$.
\begin{prop}\label{prop:induct}
Let {$K$ be the fraction field of a characteristic $0$ Dedekind domain $R$ with perfect residue fields, let}  $t,s$ be indeterminates, {$F_0:=K(t)$}, and  
$E/F_0(s)$ a finite Galois extension. 
Assume that the completion of $E/F_0(s)$ at a place  $s\mapsto s_\pi$ has Galois group $D$, inertia group $I$, and its residue extension  $N/M$ has the property  $P(D/I,\alpha,\beta)$ for some $\alpha,\beta>0$, and constant field $L$. 
{Then there exists a finite set $\mathcal{T}=\mathcal{T}(E/F_0(s), \alpha,\beta)$ of prime ideals of $R$ such that the extension $E/F_0(s)$ specializes to all tame $D$-extensions of $K_\fp$ with inertia group $I$, {for all primes $\fp\notin \mathcal T$ of $R$ which are inert in $L/K$}.}
\end{prop}
\begin{proof}
We start by choosing a ``nice" polynomial  with splitting field $E$ over $K(t,s)$. Starting with  a separable polynomial $P\in K(t,s)[x]$ with splitting field $E$, by a suitable change of the variables $s$ and $t$, we may  assume each of the following:
\begin{itemize}
\item[i)] $P$ is monic with coefficients in $R[t,s]$;
\item[ii)] {$s\mapsto s_\pi\in \oline{F}_0$ is a finite place}, and 
the minimal polynomial $Q(t,s)$ of $s_\pi$ over $F_0=K(t)$ lies in $R[t,s]$ and is monic of positive degree in both $s$ and $t$.
\end{itemize}
Indeed, rendering {$s_\pi$ finite and} the polynomials $P$ and $Q$ monic over the respective rings $R[t,s]$ and $R[t]$ is elementary. To obtain a monic polynomial $Q$ in $t$ (without touching the other conditions), substitute $s+t^k$ for $s$, for sufficiently large $k$.

Next, we introduce the constants {and exceptional primes} relevant to the proof: Let $\gamma=\gamma_Q$ be the $t$-degree of $Q$. 
{Let {$\mathcal{T}'$} be the union of the {finite} set $S_P$ provided by Lemma \ref{lem:goodprimes} with the {finite} set of primes of $R$ dividing $[E:F_0(s)]$.} 
 By Lemma \ref{lem:goodprimes}, 
 there exists a constant $d=d_P$, depending only on $P$,  such that there are at most $d$ residue classes of $s_0$ mod $\fp$
for which $\fp$ becomes a bad prime in $\mathcal S_R(P(t,s_0,x))$ for some $s_0$ of this residue, {for every prime $\fp\notin \mathcal{T}'$}. 
{Let $\mathcal{T}$ be the union of $\mathcal{T}'$ and the set of primes of $R$ of finite norm less than $(\beta+\gamma d+1)/\alpha$. Note $\mathcal{T}$ is a finite since the set of primes of bounded norm in a  characteristic $0$ Dedekind domain is finite\footnote{{There are only finitely many primes $\fp$ of norm $n>0$ in such a ring $R$ since all such primes divide $nR$ and hence occur in the factorization of $nR$.}.}}.
{As $P$ and $Q$ are picked depending only on the extension $E/F_0(s)$, so are $\mathcal{S}_1$,  $\gamma$ and $d$. Thus the set $\mathcal{T}$ depends only on $E/F_0(s)$, $\alpha$ and $\beta$.} 

Let $T_\fp/K_\fp$ be a $D$-extension with inertia group $I$ at a prime $\fp\lhd R$ {inert in $L$ and not contained in $\mathcal{T}$}. 
We claim that $T_\fp/K_\fp$ is a specialization of $E/K(t,s)$. 
Set $U_\fp = T_\fp^I$. 
Let $\mathcal T_\fp\subseteq R/\fp$ denote the set of residue classes $\oline t_0$ of $t_0$ mod $\fp$ where $U_\fp/K_\fp$ is a specialization of $N/M$ for infinitely many $t\mapsto t_0$ of  residue $\oline t_0$. Note that $|\mathcal T_\fp|\geq \alpha N(\fp)-\beta$ as the property $P(D/I,\alpha,\beta)$ holds, and $\alpha N(\fp)-\beta\geq 1$ by our choice of $\fp$. 
For each $\oline t_0\in \mathcal T_\fp$,  choose $t_0\in K_\fp$ with residue $\oline t_0$ such that: 
\begin{enumerate}
\item  $U_\fp/K_\fp$ is the residue extension of $N{K_\fp}/M{K_\fp}$ at an unramified degree $1$ prime $\fP$ of $MK_\fp$ over $(t-t_0)\lhd K_\fp[t]$;
\item  $(t-t_0)$ is not in the finite set $\mathcal S_{K_\fp[t]}(P)$ in Theorem \ref{thm:kln}, where $P$ is viewed as a polynomial in $x$ with parameter $s$ over the base ring $K_\fp[t]$;\footnote{Here we substitute $K_\fp[t]$ and $s$ for the ring $R$ and parameter $t$ in Theorem \ref{thm:kln}, respectively.}
\item $t_0$ is not one of the finitely many ramification points of the nonconstant algebraic function  $s_\pi\in \oline{K_\fp(t)}\setminus\oline{K_\fp}$ {(and in particular, not one of the roots of the discriminant of $Q(t,s)\in R[t][s]$)}. 
%
\end{enumerate}
Since $\fP$ is an unramified degree $1$ prime over $(t-t_0)\lhd K_\fp[t]$, there exists a unique $s_0\in K_\fp$ such that $v_\fP(s_\pi-s_0)>0$. 
{{C}onsider the residue extension $(EK_\fp)_{s_0}/K_{\fp}(t)$ of $EK_\fp/K_\fp(t)(s)$ at $s\mapsto s_0$.} 
By (3), we have $s_\pi-s_0\not\equiv 0$ mod $\fP^2$, and hence $v_\fP(s_\pi-s_0)=1$. 
Since in addition $(t-t_0)\notin \mathcal S_{K_\fp[t]}(P)$ by (2) and  the residue extension  of $NK_\fp/MK_\fp$ at $t\mapsto t_0$ is $U_\fp/K_\fp$ by (1),  Theorem \ref{thm:kln} implies that, at $t\mapsto t_0$, {$(EK_\fp)_{s_0}/K_\fp(t)$} has decomposition group $D$, inertia group $I$, and residue extension $U_\fp/K_\fp$.

To specialize $t$, we first note that as $s_\pi$ is a root of $Q$ and 
 $Q(t,s_0)$ is not the zero polynomial mod $\fp$ by  condition ii) above, 
 the residue $\oline t_0$ is a root of $Q(t,\oline s_0)\in R/\fp[t]$, for every value $\oline s_0\in R/\fp$. 
 Thus, running over the values $t_0$ picked above, we see that every residue $s_0$ mod $\fp$ is obtained from at most $\gamma$ values $t_0$, each with distinct residues $\oline t_0=t_0$ mod $\fp$.  
Since there are $\alpha N(\fp)-\beta$ choices for $\oline t_0$, and 
at least $(\alpha N(\fp)-\beta)/\gamma>d$ choices for $\oline s_0$, we may choose $t_0$ and $s_0$ such that $\fp$ is not a bad prime in $\mathcal S_R(P_{s_0})$ for $P_{s_0}=P(t,s_0,x)\in R[t,x]$. 
{As in addition $\fp$ is coprime to $[E:F_0(s)]$,  Theorem \ref{thm:KLN2} and Remark \ref{rem:KLN2} imply} there exist (infinitely many) $t_1\in K_\fp$, with $v_\fp(t_1-t_0)$  positive (and hence of residue $\oline t_0$) and coprime to $e$, such that the specialization of {$(EK_\fp)_{s_0}/K_\fp(t)$} at $t\mapsto t_1$ is $T_\fp/K_\fp$. 
{This shows the assertion.}
\end{proof}
\begin{rem}\label{rem:induct}
We note that the proof shows that the number of possible mod $\fp$ residues of  specialization values $(t_1,s_0)$ 
is at least $\alpha'N(\fp)-\beta'$ (or infinite if $N(\fp)=\infty$), where $\alpha',\beta'>0$ are constants that depend only on $E/K$, $\alpha$ and $\beta$. 
In particular, as soon as $\alpha'N(\fp)-\beta'>0$,  the set of possible values $(t_1,s_0)\in \mathbb{A}_{{K_\fp}}^2$ is Zariski dense. Namely, there are infinitely many ``good" values $t_1$, and for each such value infinitely many good values $s_0$.
\end{rem}

Finally, for a number field $K$, 
we show that a $K$-regular extension has the property $P(C,\alpha,\beta)$ for certain cyclic subgroups $C$, and constants $\alpha,\beta>0$. 
\begin{lem}
\label{lem:spec_main}
Let $K$ be a number field and $M/K(t)$ a finite $K$-regular Galois extension. 
Let $N/M$ be a Galois extension with constant field $L=N\cap \oline K$, and $C\leq \Gal(N/M)$ a cyclic subgroup which projects onto $\Gal(LM/M)$.
Then $N/M$ has the property $P(C,\alpha,\beta)$ for some constants $\alpha,\beta>0$ depending only on $N/K$. 
\end{lem}
\begin{proof}
Let $\fp$ be a prime of $K$ which is inert in $L$. First note that since $M$ is $K$-regular,  $M K_\fp$ is also $K_\fp$-regular (since the inertia groups in $\Gal(MK_\fp/K_\fp(t))$ still generate $\Gal(M/K(t))$). Similarly as $L$ is the constant field of $N$,   the constant field of $N K_\fp$ is $L K_\fp$. 
Thus, $\Gal(NK_\fp/MK_\fp)$ is isomorphic to $\Gal(N/M)$. 
Let $\Omega_N/K(t)$ be the Galois closure of $N/K(t)$.
Then in particular, the composition of the restriction map with the above isomorphism gives an epimorphism $p:\Gal(\Omega_N K_\fp/MK_\fp)\ra \Gal(N/M)$. 
Letting $x$ be a generator of $C$,  choose  $\widehat{x}\in p^{-1}(x)$. 
Then by Theorem \ref{thm:debesghazi_ext} and Remark \ref{rem:debesghazi_ext},  there exist at least $\alpha N(\fp)-\beta$  residue classes  $\oline t_0$ mod $\fp$, and infinitely many places $t\mapsto t_0\in K_\fp$ of each residue, at which $\Omega_N K_\fp/ K_\fp(t)$ specializes to an unramified extension of $K_\fp$ with group $\langle\widehat{x}\rangle$, for  $\alpha,\beta>0$ depending only on $N/K$. 
In particular, $MK_\fp/K_\fp(t)$ specializes to the trivial extension, and $NK_\fp/MK_\fp$ specializes to an extension with group $p(\langle\widehat{x}\rangle)=C$. 
\end{proof}

The following corollary combines Proposition \ref{prop:induct} and Lemma \ref{lem:spec_main} to show that the extension constructed in Corollary \ref{thm:construct-2}  specializes to every $D$-extension of $K_\fp$ with inertia group $I$, {for all but finitely many primes $\fp$ of $K$}. 

 For an ID pair {$\pi = (I,D)$} with $e:=|I|$, 
let $\eta_\pi:D\ra   \Gal(K(\mu_e)/K)$ be the associated map in Lemma \ref{lem:id_pairs}, and let $V_\pi$ be its image. 
\begin{cor}\label{cor:2-dim}
Let $\pi=(I,D)$ be an ID pair over a number field $K$, and $e:=|I|$. 
Let $E/K(t,s)$ be a Galois extension whose completion at a place $s\mapsto s_\pi$ is an extension $E_\pi/F_\pi$
with residue extension $N/M$ satisfying conditions (a)-(c) of Corollary \ref{thm:construct-2}.  
Then there exists $\varepsilon>0$, depending only on $E/K$, 
such that every tame $D$-extension of $K_\fp$ with inertia group $I$ is a specialization of $E/K(t,s)$,  for every prime $\fp$ of $K$ with $N(\fp)>\varepsilon$. 
\end{cor}
\begin{proof}
Let $\fp\in P_K(\pi)$, that is, a prime for which there exists a $D$-extension of $K_\fp$ with inertia group $I$. By Lemma \ref{lem:id_pairs}.(c), the Frobenius element of $\fp$ in $K(\mu_e)/K$ is a generator of $V_\pi$. 
In particular, a prime $\fP$ of $K_0:=K(\mu_e)^{V_\pi}$ over $\fp$ is of degree $1$, and is inert in $K(\mu_e)$. 
Note that $K_0\subseteq  (K_0)_\fP\cong K_\fp$. 

By construction,  $N/M$ has cyclic Galois group $D/I$, and  its {constant} extension is $K(\mu_e)/K_0$. Thus, $N/M$ has the property $P(D/I,\alpha,\beta)$ over $K_0$ for some $\alpha,\beta>0$ depending only $N/K$, by Lemma \ref{lem:spec_main}.  As $\fP$ is inert in $K(\mu_e)$, Proposition \ref{prop:induct} 
gives a constant $\varepsilon>0$,
depending only on $E/K$ (since the choice of $\alpha,\beta$ depends only on $E/K$) with the following property: Every tame local $D$-extension with inertia group $I$ of $K_\fp$ is a specialization over $K_\fp$ of $EK_0/K_0(t,s)$ and hence of $E/K(t,s)$, when $N(\fp)>\varepsilon$. {Here, $\varepsilon$ is the largest norm of a prime in the set $\mathcal S$ of Proposition \ref{prop:induct}.} 
\end{proof}

In case the ID-pair $(I,D)$ is split and $C_D(I)=I$, it is even {simpler} to deduce that the extension in Corollary \ref{cor:construct-1} specializes to every $D$-extension with inertia group $I$. 
Let $\eta_\pi$ and $V_\pi$ be as in Lemma \ref{lem:id_pairs}.  
\begin{cor}\label{cor:1dim2} 
Let $\pi=(I,D)$ be a split ID pair over a number field $K$ with $C_D(I)=I$. 
Let $e:=|I|$, and $\eta_\pi:D\ra V_\pi$ be as above. 
Let $E/K(t)$ be a Galois extension whose completion at a place $t\mapsto t_\pi$ is an extension $E_\pi/F_\pi$ with residue extension $K(\mu_e)/K(\mu_e)^{V_\pi}$.
Then 
every tame $D$-extension of $K_\fp$ with inertia group $I$ is a specialization of $E/K(t)$,  for every prime 
$\fp\notin\mathcal S_{0}(E/K(t))$.
\end{cor}
\begin{proof}

As in Corollary \ref{cor:2-dim}, let $\fp\in P_K(\pi)\setminus \mathcal S_{0}(E/K(t))$,  
so that a prime $\fP$ of $L:=K(\mu_e)^{V_\pi}$ over $\fp$ is of degree $1$, is inert in $K(\mu_e)$, and has completion $L_\fP\cong K_\fp \supseteq L$. 
Let $T$ be a tame $D$-extension of $K_\fp\cong L_\fP$ with inertia group $I$. 
As $I=C_D(I)$, Lemma \ref{lem:id_pairs}.(1) implies that $T^I\cong L_\fP(\mu_e)$.
As $\fP$ is inert in $L(\mu_e)/L$, the extension $T^I/L_\fP$ is the compositium of $L(\mu_e)/L$ with $L_\fP$.  
Since $\fp\notin \mathcal S_{0}{(E/K(t))}$, and since $EL_\fP/L_\fP(t)$ is a $D$-extension with inertia group $I$ and residue  $L_\fP(\mu_e)$, there exist (infinitely many) $t_0\in K_\fp\cong L_\fP$ 
such that the specialization of $EL_\fP/L_\fP(t)$ at $t\mapsto t_0$ is  $T/K_\fp$, by Theorem \ref{thm:KLN2}. 
\end{proof}

\section{Proof of the main theorem and further rationality analysis}\label{sec:proof-main}

\subsection{Proof of part (2a) of the main theorem}
Let $\Pi=\Pi_{G,K}$ denote the set of all ID pairs $(I,D)$ over $K$ with $D\leq G$. Let $\eta_\pi:D\ra \Gal(K(\mu_e)/K)$ be the map associated to $\pi\in \Pi$ in Lemma \ref{lem:id_pairs}, and $V_\pi$ its image. 
By Corollary \ref{thm:construct-2}, we may choose distinct places $s\mapsto s_\pi$, $\pi\in \Pi$ of $K(t)(s)$, and extensions $E_\pi/F_\pi$ of the completion $F_\pi$ of $K(t,s)$ at $s\mapsto s_\pi$ satisfying the following. For $\pi=(I,D)$, the extension $E_\pi/F_\pi$ has Galois group $D$, inertia group $I$, and  residue extension $N/M$ satisfying: 
\begin{enumerate}
\item[(a)] $M/K(\mu_e)^{V_\pi}(t)$ is  finite $K(\mu_e)^{V_\pi}$-regular, where $e:=|I|$; 
\item[(b)] the  constant field of $N$ is $K(\mu_e)$;
\item[(c)] $N^{C} \cong M(\mu_e)$, where $C:=C_D(I)$ acts via the natural quotient $C\ra C/I$. 
\end{enumerate}
Since $G$ has a generic extension over $K$, \cite[Theorem 5.8]{Sal} gives a $G$-extension $E/K(t,s)$ with decomposition group conjugate to $D$ and completion $E_\pi/F_\pi$ at $s\mapsto s_\pi$, for every $\pi=(I,D)\in \Pi$. 
 By Corollary \ref{cor:2-dim}, every tame $D$-extension  $T_\fp/K_\fp$ with inertia group $I$ is a specialization of $E/K(t,s)$ at some $t\mapsto t_\pi'$, $s\mapsto s_\pi'$, for every $\pi\in \Pi$ and prime $\fp$ of $K$ with $N(\fp)>\varepsilon$ for  a constant $\varepsilon$ depending only on $E/K$.

\subsection{Proof of part (1) of the main theorem}
Recall from Section \ref{sec:notation} that  a $G$-Galois extension $\mathcal L/K$ of \'etale algebras is a specialization of a $G$-extension $E/F$ of fields at a place $\nu$ if and only if the field extension $L/K$, underlying $\mathcal L/K$,  is the residue extension at $\nu$.  
Thus, the  local dimension $\ld_K(G)$ (resp.\  parametric dimension $\pd_K(G)$) of $G$ over $K$ is the minimal integer $d\geq 0$ for which there exist $G$-extensions $E_i/F_i$, $i=1,\ldots,r$, such that every field $H$-extension of $K_\fp$ (resp.\ of $K$) for $H\leq G$, is a specialization of some $E_i/F_i$, for all but finitely many primes $\fp$. 
\begin{lem}\label{lem:subgroup}
Assume $G$ has local (resp., parametric) dimension $\le d$ over $K$. Then the same holds for any subgroup of $G$.
\end{lem}
\begin{proof}
Let $E/F$ be a $G$-extension of function fields such that $K$ is the constant field of $F$.
For $H\leq G$, let $L$ be an overfield of $K$, and assume that $\nu$ is a place of $FL$ at which the residue extension of $EL/FL$ is a field $H$-extension $M/L$. This  implies that there is a place $\omega$ of $E^HL$ of degree $1$ over $\nu$ such that the residue of $EL/E^HL$ at $\omega$ is $M/L$.
The assertion now follows by choosing $L=K$ for parametric dimension or $L=K_\fp$ for primes $\fp$ of $K$ for local dimension.
 \end{proof}

In particular, since any finite group embeds as a subgroup into a group possessing a generic extension (such as a suitable symmetric group), Lemma \ref{lem:subgroup} and part (2a) of the main theorem  imply  part (1) of the main theorem. 

The  argument in Lemma \ref{lem:subgroup} does not apply in the Hilbert-Grunwald dimension setting, since it is not clear how to construct $K$-rational places of $E^H$ for $H\leq G$. However, it  yields a weak version of the Hilbert--Grunwald property, see Corollary \ref{cor:hgd}. 

\subsection{Proof of part (2b) of the main theorem} 
\label{sec:proof_grunwald}
Part (2b) of the main theorem is derived from (2a) using:
\begin{lem}
\label{cor:grunwald}
Let $K$ be a number field, $G$ a finite group and 
$E/F$  a $G$-extension such that $F$ is purely transcendental over $K$, and such that for all primes $\fp$ of $K$ outside of some finite set $T$, every $G$-extension of $K_{\fp}$ occurs as a specialization of $E/F$ at some place away from the branch locus of $E/F$.
Then every Grunwald problem for $G$ away from $T$ has a solution which is a field extension with Galois group $G$, and contained in the set of specializations of $E/F$.
\end{lem}
\begin{proof}
First, by enlarging $S$,  add to the Grunwald problem  (unramified) $C$-extensions $L^{(\fp_C)}/K_{\fp_C}$ over distinct primes $\fp_C\notin S\cup T$ where $C\leq G$ runs over cyclic subgroups.

Suppose $F=K({\bf t})$, where ${\bf t}=(t_1,\ldots,t_r)$ are indeterminates. 
 Let $P({\bf t},x)$ be a defining polynomial of $E/F$, separable in $x$, and let ${\bf t}\ra {\bf t}_{\fp}\in K_\fp^r$, $\fp\in S\setminus T$ be places at which $E/F$ specializes to prescribed extensions $L^{(\fp)}/K_{\fp}$. 
We claim that there are in fact infinitely many choices for ${\bf t}_p\in K_\fp^r$  with the above property for  $\fp\in S\setminus T$. 
 
As in Section \ref{sec:twists}, letting $f:X\ra Y$ be a $G$-cover over $K$ with function fields $E/F$ and $\psi_\fp:G_{K_\fp}\ra G$ an epimorphism representing $L^{(\fp)}/K_\fp$, the point 
 ${\bf t}_{\fp}$ gives rise to a $K_\fp$-rational point on $\tilde X$, where $f^{\psi_\fp}:\tilde X\ra\mP^1$ is the twist of $f$ by $\psi_\fp$.
Since $K_\fp$ is an ``ample" field, a.k.a.\ ``large" field, every smooth absolutely irreducible variety with a point has infinitely many points. These yield infinitely many possibilities for ${\bf t}_\fp$, as claimed. 
By possibly replacing ${\bf t}_\fp$, we henceforth assume $P({\bf t}_{\fp},x)$ is separable. 

This makes Krasner's lemma applicable, yielding full mod-$\fp^m$ residue classes, for some $m>0$, of values ${\bf t}_{\fp} \in K_\fp^r$ with the same property. 
 By the Chinese remainder theorem, we pick (infinitely many) ${\bf t}_0\in K^r$ in the above residue classes, so that $E_{{\bf t}_0}/K$ has the prescribed completions at all $\fp\in S$. 
 Finally, since we extended our Grunwald problem, the Galois group of $E_{{\bf t}_0}/K$ must intersect every conjugacy class of cyclic subgroups in $G$, and thus equal $G$ by Jordan's theorem. 
\end{proof}
In particular, if there exists a single extension $E/F$, with $F$ purely transcendental of transcendence degree $\ld_K(G)$, that specializes to $G$-extensions of $K_\fp$ for $\fp$ away from a finite set,  then $\hgd_K(G) = \ld_K(G)$. 

Part (2b) of the main theorem now follows by choosing $E/F$ as in part (2a).
\begin{rem}\label{rem:IGP}
Note that it is possible that the equality $\hgd_K(G)=\ld_K(G)$ holds in general, however already $\hgd_K(G)<\infty$ implies a positive answer to the IGP for $G$ over $K$. Indeed, by adding suitable local extensions to a Grunwald problem, one may force its solutions to be fields as in the proof of Lemma \ref{cor:grunwald}.
\end{rem}
\begin{rem}\label{rem:WWA}
The existence of a generic extension is used in the proof of the main theorem  in order to find a $G$-extension $E/K(t,s)$ whose completions at $s\mapsto s_\pi$ are given extensions $E_\pi/F_\pi$, for all ID pairs $\pi\in \Pi$ for $G$ over $K$. 

Embedding $G$ into $SL_n$, and letting $X:= SL_n/G$, it is well known that $SL_n\ra X$ is a versal $G$-torsor\footnote{One may also consider the map $\mathbb A^n\ra\mathbb A^n/G$, and take a smooth open subset of $\mathbb A^n/G$ to form a $G$-torsor.}. 
In this setting, the property required for part (2b) of the main theorem is the following weak approximation property. Namely, the restriction
$$ X(F_0(s)) \ra \prod_{\pi\in \Pi}X(F_\pi), $$
has to have dense image, where $F_0:=K(t)$, and $F_\pi$ is the completion of $F_0(s)$ at a place $s\mapsto s_\pi$. 
This weak approximation property clearly holds if $X$ is rational,
but fails in general in view of Remark \ref{rem:finite}. 
\end{rem}

\subsection{Some strengthenings}\label{sec:ext}
We note some strengthenings of the assertions of the main theorem, which directly follow from the proof above.
In regards to the Hilbert--Grunwald dimension, the proof of Lemma \ref{lem:subgroup}  and part (2b) of the main theorem  give the following for every finite group $G$ and number field $K$:
\begin{cor}\label{cor:hgd}
There exists a constant $d=d_{G}$, a $G$-extension $E/F$ with $F$ of transcendence degree $2$, and a finite set $T$ of primes of $K$ such that every Grunwald problem $L^{(\fp)}/K_\fp$, $\fp\in S\setminus T$, is solvable over an overfield  $K'$ of degree $[K':K]\leq d$ in the following sense:
For every prime $\fp\in S\setminus T$, there is a {prime} $\fp'$ of $K'$ with $K'_{\fp'}\cong K_\fp$, such that the Grunwald problem $L^{(\fp)}/K'_{\fp'}, \fp\in S\setminus T$ is solvable within the set of  $K'$-rational specializations of $E/F$. 
\end{cor}

\begin{proof}
Embedding $G$ into a group $\Gamma$ possessing a generic extension, let  $E/K(t,s)$ be the $\Gamma$-extension provided by part (2b) of the main theorem, and set $F=E^G$. 
As in the proof of Lemma \ref{lem:subgroup}, the property that $E/K(t,s)$ specializes to every $H$-extension $L^{(\fp)}/K_\fp$ for $H\leq G$, implies that there are (distinct) $K_\fp$-rational places $\nu(\fp)$, $\fp\in S\setminus T$ of $K_\fp(t,s)$, and  places $\omega(\fp)$ of $FK_\fp$ of degree $1$ over $\nu(\fp)$, such that the residue of $E/F$ at $\omega(\fp)$ is $L^{(\fp)}/K_\fp$, for every $\fp\in S\setminus T$. As in the proof of Lemma \ref{cor:grunwald}, there exists a $K$-rational place $\nu$ of $K(t,s)$ approximating $\nu(\fp), \fp\in S\setminus T$,  at which  the residue extension $L/K$ of $E/F$ has completion $L^{(\fp)}/K_\fp$, for all $\fp\in S\setminus T$. 
Let  $\omega$  be a place of $F$ over $\nu$ and  $K'$ denote the residue field at $\omega$. Since the decomposition group of $L/K$ at $\fp$ is conjugate to a subgroup of $G$, there is a degree $1$ prime $\fp'$ of $K'$ over $\fp$ such that $K_{\fp'}\cong K_\fp$, for each $\fp\in S\setminus T$. Thus $L/K'$ is a solution to our Grunwald problem. However, now $[K':K]$ 
 is  bounded merely by $d_G:=[F:K(t,s)]$.
\end{proof}

The proof of the main theorem also shows the following stronger conclusion:
\begin{cor}\label{cor:ext}
The $G$-extension $E/F$, constructed in the main theorem,  satisfies the following for each finite extension $K'\supseteq K$:
\begin{enumerate}
\item for each prime $\fp'$ of $K'$ outside of some finite set $T=T(E,K')$, every $G$-extension of $K'_{\fp'}$ occurs as a specialization of $E/F$ over $K'_{\fp'}$;
\item If $G$ admits a generic extension over $K$, every Grunwald problem $L^{(\fp')}/K'_{\fp'}$, $\fp'\in S\setminus T$ has a solution within the set of specializations of $E/F$ over $K'$. 
\end{enumerate}
\end{cor} 
\begin{proof}
We first claim that every ID pair $\pi=(I,D)$ for $G$ over $K'$ is also an ID pair for $G$ over $K$, by  Lemma \ref{lem:id_pairs}:
 Observe that the composition of the map $\eta'_\pi:D\ra\Gal(K'(\mu_e)/K')$ with the restriction $\Gal(K'(\mu_e)/K')\ra \Gal(K(\mu_e)/K'\cap K(\mu_e))$ is $\eta_\pi:D\ra \Gal(K(\mu_e)/K)$. By Chebotarev's theorem, there exist (infinitely many) primes $\fp$ of $K$ of norm $q$ such that $\langle\sigma_{q,e}\rangle=\Im(\eta_\pi)$. Thus, the ``if" direction of Lemma \ref{lem:id_pairs}.(2) implies that $(I,D)$ is an ID pair over $K$, as claimed. 
 
 Now the base change $EK'/FK'$ still has a completion $E_\pi/F_\pi$ satisfying the conditions of Corollary \ref{thm:construct-2}, for every ID pair $\pi$ for $G$ over $K'$. Furthermore, it is purely transcendental over $K'$ if $G$ has a generic extension over $K$. Thus the proof of the main theorem applies over $K'$ for the extension $EK'/FK'$.
 \end{proof}

 Finally note that the notion of ``arithmetic dimension" \cite{ONeil} is similar to parametric dimension but allows ``finite base change". That is, it measures the number of parameters required to parametrize all \'etale algebra $G$-extensions of $K'$, where $K'$ runs over all finite extensions of $K$. The above corollary suggests that the ``local" version of arithmetic dimension is then also $2$. Thus in similarity to Remark \ref{rem:LGP} the study of arithmetic dimension in this context also concerns the extent to which local global principles fail.

\subsection{Allowing finitely many parametrizing extensions}\label{sec:finitely}
Let {$K$ be a number field} and $\hat T$ the union of the infinite places of $K$ with the finitely many places corresponding to primes in the set $T$ from the main theorem. 
As there are only finitely many Galois extensions of completions at $\nu\in \hat T$ with Galois group a subgroup of $G$, we may cover these extensions by adding to $E/F$ finitely many extensions:
\begin{cor}\label{cor:many-lifts}
There exist finitely many $G$-extensions $E_i/F_i$, $i=1,\ldots,r$, over fields $F_i$ of transcendence degree $2$ over $K$, such that every $H$-extension of a completion of $K$, for every $H\leq G$, is a specialization of some $E_i/F_i$. 
\end{cor}
\begin{proof}
We claim that every field $H$-extension of a completion at $\nu\in \hat T$, for every $H\leq G$, is the completion of some $G$-extension $L/K'$, where $K'/K$ is finite.  The resulting list of extensions $E_i/F_i$ then consists of the extension $E/F$ of the main theorem, and finitely many extensions $L(t,s)/K'(t,s)$ coming from the finitely many extensions $L/K'$ given by the claim.

To see the claim,  fix an $H$-extension $L^{(\nu)}$ of the completion of $K$ at $\nu$, for $\nu\in \hat T$ and $H\leq G$. Embed $G\leq S_n$. Since $S_n$ has a generic extension over $K$, {as in the proof of Lemma \ref{lem:subgroup}}, there exists an $S_n$-extension $L$ of $K$ with completion $L^{(\nu)}$ at $\nu$. 
Set $K'=L^G$. 
Since the decomposition group of $L/K$ at $\nu$ is (conjugate to) a subgroup of $G$, there exists a degree $1$ prime $\omega$ of $K'$ over $\nu$. The completion of $L/K'$ at $\omega$ is isomorphic to the completion of $L/K$ at $\nu$, so that $L/K'$ is the desired extension for the claim. 
\end{proof}
It is conjectured that every $H$-extension of a number field $K$ is a specialization of a $K$-regular $H$-extension of $K(t)$, a property also known as the Beckmann--Black lifting property, cf.\ \cite{Deb-BB}.   In view of this conjecture, it seems likely that the extensions $E_i/F_i$, constructed in Corollary 
\ref{cor:many-lifts}, can be replaced by $K$-regular extensions {with purely transcendental base field:}

\begin{rem}\label{rem:prediction}
{One remaining open question is whether the extensions $E_i/F_i$ as in Corollary \ref{cor:many-lifts} may in fact be chosen {$K$-regular} with $F_i = K(t,s)$ purely transcendental (for all $i$). Of course, for groups $G$ possessing a generic extension over $K$, this is already achieved by the Main Theorem. For general groups $G$,}
it suffices to construct, for each ID pair  $\pi=(I,D)$ of $G$ over $K$, a $G$-extension of $K(t,s)$ which specializes to all local $D$-extensions with inertia group $I$ over all but finitely many primes. As the weak approximation property in Remark \ref{rem:WWA} indicates, to do so it suffices to show $E_\pi/F_\pi$ is the completion at $s\mapsto s_\pi$ of some {$K$-regular} $G$-extension $E(\pi)/K(t,s)$. 

{Without the $K$-regularity requirement and upon allowing 
 $E(\pi)/K(t,s)$ to be an extension of \'etale algebras},  this can often be {achieved even for groups which do not necessarily possess a generic extension, as the following example demonstrates.}
\end{rem}
\begin{exam}
Let $G$ be an $\ell$-group, and $K$ a number field containing the $|G|$-th roots of unity. 
Let $\pi=(I,D)$ be an ID pair for $G$ over $K$, and $E_\pi/F_\pi$ the constructed completion of $K(t,s)$ at $s\mapsto s_\pi$ in Corollary \ref{thm:construct-2}. Then there exists a $D$-extension of $K(t,s)$ with completion $E_\pi/F_\pi$ at $s\mapsto s_\pi$, inducing the desired $G$-extension of \'etale algebras $E(\pi)/K(t,s)$.

Indeed, when $G$ is a $\ell$-group and $K$ contains the $|G|$-th roots of unity, Remark \ref{rem:p-power} shows that $s\mapsto s_\pi$ can be chosen $K$-rational {and finite}.
Thus, for an ID-pair $(I,D)$, the field  $E_\pi$ is of the form $M((s-s_\pi))(\sqrt[e]{\alpha(s-s_\pi)})$, where $e=|I|$, $M/K(t)$ is a $(D/I)$-extension, and $\alpha\in M$. In this case $M(s,\sqrt[e]{\alpha(s-s_\pi)})/K(t,s)$ is a $D$-extension   whose completion at $s\mapsto s_\pi$ is $E_\pi/F_\pi$.
\end{exam}

\appendix

\section{Parametric dimension {and related notions}}\label{sec:parametric}
The following example was noticed in \cite[Theorem 3.4]{ONeil}. It shows that $\pd_\mQ(G)$ may be strictly smaller than $\ed_\mQ(G)$,  based on the Hasse--Minkowski theorem: 
\begin{exam}
Let $G=(\Z/2)^5$. It is well known that $\ed_\mQ(G)=5$. We claim that $\pd_\mQ(G)\leq 4$. Namely, we consider the extensions\footnote{The number of extensions is clearly not minimal, and it is interesting to see if it is possible to parametrize all $G$-extensions using a single extension of transcendence degree {$4$}.} $E_{i,\eps}:=\mQ(\sqrt{t_1},\ldots,\sqrt{t_4},\sqrt{\eps(t_1+\cdots+t_i)})$, $i\leq 4$, $\eps\in \{1,-1\}$ of $F:=\mQ(t_1,\ldots,t_4)$, and claim that every $G$-extension $L/\mQ$ is a specialization of at least one of these extensions. 
It is well known that $L$ is of the form $\mQ(\sqrt{a_1},\ldots,\sqrt{a_5})$ for some $a_i\in \mQ^\times$, $i=1,\ldots,4$. (Note that $a_i=0$ may be replaced by $a_i=1$ without changing the extension).   

We consider two cases according to whether all $a_i$'s are positive or at least one of them is not. In the former case set $\eps=+1$, and in the latter $\eps=-1$. 
By the Hasse--Minkowski theorem, the quadratic form $Q_{\eps}(X_1,\ldots,X_5):=a_1X_1^2+\cdots +a_4X_4^2-\eps a_5X_5^2$ has a nontrivial solution $(x_1,\ldots,x_5)$. By possibly reordering the $a_i$'s, we may assume $x_5$ and $x_1,\ldots,x_i$  are nontrivial for some  $1\leq i\leq 4$. Then the specialization of $E_{i,-\eps}$, at $t_j\mapsto a_jx_j^2$ for $j\leq i$, and $t_j\mapsto a_j$ for $i<j\leq 4$,  is 
$\mQ(\sqrt{a_1},\ldots,\sqrt{a_4},\sqrt{-\eps\sum_{j=1}^ia_jx_j^2})$. The last field is $L$ since $-\eps\sum_{j=1}^i a_jx_j^2\in \mQ^\times$ and $a_5\in \mQ^\times$ are equivalent mod $(\mQ^\times)^2$ by our choice of $x_1,\ldots,x_i$ and $x_5$. 
\end{exam}  
Allowing finitely many extensions in the definition of parametric dimension rather than only one seems natural in analogy to the fact that the finite union of $d$-dimensional varieties is still a $d$-dimensional object. 
The following example indicates that allowing more than one extension (but still a finite number) 
 might\footnote{The example considers extensions of rational function fields. Since we do not require rationality in the definition of $\pd_K(G)$, it remains to be seen whether there is an example of a number field $K$ and group $G$ such that  $\pd_K(G)$ would increase upon demanding $r=1$ in the definition.} be necessary 
over number fields. 
\begin{exam}\label{exam:finite}
Set $G=\mathbb Z/8$, and $K=\mQ(\sqrt{17})$. Then there exist two primes $\fp_1,\fp_2$ of $K$ and Galois extensions $L^{(\fp_i)}/K_{\fp_i}$ with groups $H_i\leq G$, such that there is no $G$-extension $E/K(t_1,\ldots,t_r)$, for any $r\geq 0$, that specializes to both $L^{(\fp_i)}/K_{\fp_i}$, $i=1,2$. On the other hand,  it is known \cite{MS} that there exist finitely many polynomials over $K(t_1,\ldots,t_5)$ which parametrize all $G$-extensions.

Indeed, in a joint work with D.~Krashen\footnote{This example can be viewed as  an explicit version of  \cite[Proposition A.3]{CT2}.}, 
we construct   Galois extensions $L^{(\fp_i)}/K_{\fp_i}$ over two distinct (even) primes $\fp_1,\fp_2$ of $K$ with groups $H_i\leq G$ (resp.\ $G$-extensions $L_i/K$), $i=1,2$ such that there is no $G$-extension $E/K(t)$ that specializes to both $L^{(\fp_i)}/K_{\fp_i}$ (resp.\ $L_i/K$), $i=1,2$.
If there was a single extension $E/K({\bf t})$, ${\bf t}=(t_1,\ldots,t_r)$ which specializes at ${\bf t}\mapsto {\bf t}_i$ to  $L^{(\fp_i)}/K_{\fp_i}$, $i=1,2$ (resp.~to every extension $L/K$) with Galois group a subgroup of $G$, then by finding a line $t_i=t_i(t)$, $i=1,\ldots,r$ passing through ${\bf t}_1$ and ${\bf t}_2$, we would obtain an extension $E_1/K(t)$ that specializes to $L^{(\fp_i)}/K_{\fp_i}$ (resp.~$L_i/K$), $i=1,2$, obtaining a contradiction. 
\end{exam}
Note that the above example, as well as the following remark,  can be interpreted in terms of rational connectedness on the quotient {$X:=SL_n/G$} for $G\leq S_n$, {cf.\ Remrark \ref{rem:WWA},} or more precisely $R$-equivalence on $X$ over $K$: 
The $K$-rational points on $X$ which are unramified under the projection $SL_n\ra X$ correspond to $G$-extensions of $K$, while  a $K$-rational curve $\mathbb P^1\subseteq X$ corresponds to a $G$-extension of $K(t)$. Thus, two $G$-extensions of $K$ which are not specializations of a $G$-extension of $K(t)$ correspond to two $K$-rational points on $X$ which cannot be connected by a rational curve $\mP^1$ over $K$, in similarity to the notion of $R$-equivalence or rational connectedness.

\begin{rem}\label{rem:finite}
{While we know that function field extensions of transcendence degree $1$ in general do not suffice to parameterize all $G$-extensions of a given number field $K$ (even when allowing finitely many function fields), it is still possible, in particular in view of \cite[\S 10]{CT3} and \cite[\S 2.4]{Wit},  that there is a uniform bound} $b=b_{G,K}$, depending only {on} the group $G$ and the number field $K$, satisfying the following {two assertions. (1) For every
finitely many Galois extensions $L_i/K$, $i=1,\dots, r$ with Galois group a subgroup of $G$, there exist Galois extensions $E_1/K(t), \dots, E_b/K(t)$ with Galois group $G$ such that each of the extensions $L_i/K$ is a specialization of some $E_j/K(t)$. (2) In the same way, for every
finitely many primes $\fp_i$, $i=1,\ldots,r$, and Galois extensions $L^{(\fp_i)}/K_{\fp_i}$, $i=1,\ldots,r$  with group a subgroup of $G$, there exist Galois extensions $E_1/K(t), \dots, E_b/K(t)$ of group $G$ (with the extensions themselves, but not their number, possibly depending on the chosen primes {and extensions $L^{(\fp_i)}/K_{\fp_i}$}) such that each of the extensions $L^{(\fp_i)}/K_{\fp_i}$ 
is a specialization of some {$E_jK_{\fp_i}/K_{\fp_i}$} for $i=1,\ldots,r$. These properties, if fulfilled, would constitute a generalization of the conjectured Beckmann-Black property, {cf.\ the discussion before Remark \ref{rem:prediction}.}} 

\end{rem}

Finally, the following remark {connects the question whether $\pd_K(G)$ is bounded from above by $2$ with certain local global principles; more precisely, $\pd_K(G)\leq 2$ whenever the principles hold.}  First note that we expect the following ``morphisms" version of part (1) of the main theorem 
also holds for every finite group $G$ and number field $K$:
There exists a a morphism $\varphi:G_F\ra G$, where  $F$ is of transcendence degree $2$ over $K$, such that the  set of specializations $\varphi_{t_0}:G_K\ra G$ of $\varphi$ surjects onto $\Hom(G_{K_\fp},G)$ under the restriction map, for every prime $\fp$ away from a finite set $T$. 
\begin{rem}\label{rem:LGP}
Let $E/F$ and $T$ be as in the main theorem. In the setting of Section \ref{sec:twists}, the constructed extension $E/F$ in the main theorem  can be replaced by a dominant Galois morphism   $f:X\ra Y$ of smooth irreducible $2$-dimensional varieties defined over $K$ with Galois group $G$. 
Let $\varphi: \pi_1(Y') \to G$ be an epimorphism representing $f$. 
For $\psi:G_K\ra G$, denote by $\psi_\fp:G_{K_\fp}\ra G$ its restriction to $G_{K_\fp}$, and consider the twists $f^\psi:\tilde X^\psi\ra Y$ and $f^{\psi_{\fp}}:\tilde X^{\psi_\fp}\ra Y$ (which is the base change of $f^\psi$ to $K_\fp$).

We first claim that the local global principles:
$$\prod_{\fp\notin T}\tilde X^\psi(K_\fp)\neq \emptyset  \Rightarrow \tilde X^\psi(K)\neq\emptyset,$$ 
for all $\psi:G_K\ra G$ imply that every $G$-extension of $K$ is a specialization of $E/F$.
 
Indeed, letting $\psi:G_K\ra G$ be a morphism corresponding to a $G$-extension  $L/K$,  the above version of the main theorem implies that the variety $\tilde X^\psi$ has a $K_\fp$-rational point for  all $\fp\notin T$. 
Thus, the above local global principle implies that $\tilde X$ has a $K$-rational point, and hence by the main property of $f^\psi$, that $\psi$ is a specialization of $f$. This in turn implies that $L/K$ is a specialization of $E/F$, as claimed.

Moreover, when choosing $\psi$ whose completions $\psi_\nu$  are known to appear as specializations of $\varphi$, where $\nu$ runs through places in $T$ and infinite places, then $X^\psi$ has a rational point over the completion at $\nu$. The above local global principle then {trivially} takes the form:
$$\prod_{\nu}\tilde X^\psi(K_\nu)\neq \emptyset  \Rightarrow \tilde X^\psi(K)\neq\emptyset,$$ 
 where $\nu$ runs over all places of $K$.

Therefore, in similarity to \cite{KL2}, the expectation that $\pd_K(G)>2$ for many groups $G$ therefore implies a failure of the local global principle. 
\end{rem}

\section{Groups of local dimension $1$} \label{sec:1-dim}

The following result, {which follows largely from \cite{KLN},} shows that the upper bound $\ld_K(G)\leq 2$ is sharp {for many groups $G$:}
\begin{thm}
\label{thm:locdim_not1}
Let {$K$ be a number field} and $G$ be a finite group containing at least one noncyclic abelian subgroup. Then $G$ has local dimension at least $2$ over $K$.
\end{thm}

{The proof of Theorem \ref{thm:locdim_not1} rests on the following result, which is shown in \cite{KLN}.} 
\begin{prop}
\label{prop:locdim_not1}
Let $K$ be a number field, $G$ a finite group, and $E_i/K(t)$, $i=1,\dots, r$ finitely many Galois extensions with group $G$. Then there exists a finite extension $L/K$ such that, for all primes $\fq$ of $K$ which split completely in $L$, for all $i\in \{1,\dots, r\}$, and for all non-branch points $t_0\in \mathbb{P}^1(K_\fq)$ of $L$, the specialization $(E_i\cdot K_\fq)_{t_0}/K_\fq$ is cyclic. In particular, if $G$ contains a minimal noncyclic abelian subgroup $A\cong C_p\times C_p$ (for some prime $p$), then for all primes $\fq$ of $K$ which split completely in $L(\zeta_p)$, the field $K_\fq$ possesses an $A$-extension, whereas no specialization over $K_\fq$ of any $E_i/K(t)$ has group $A$. 
\end{prop}

\begin{rem}
In \cite{KLN},  these Galois extensions $E_i/K(t)$ were assumed $K$-regular. That assumption is however not necessary for the assertions given here. Indeed, if $K'$ is the full constant field of a $G$-extension $E/K(t)$, then the first assertion {of Proposition \ref{prop:locdim_not1} holds 
for the $K'$-regular extension $E/K'(t)$, yielding a field $L\supseteq K'$}. If  $\fq$ is a prime of $K$ completely split in $L$, then any prime $\fq'$ extending $\fq$ in $K'$ is of degree $1$ and split in $L$, whence $(E\cdot  K_\fq)_{t_0}/K_\fq$ must be cyclic for all non-branch points $t_0\in K_\fq (\cong K'_{\fq'})$.
\end{rem}

\begin{proof}[Proof of Theorem \ref{thm:locdim_not1}]
Proposition \ref{prop:locdim_not1} implies that no finite set of $G$-extensions $E_i/K(t)$, $i=1,\ldots,r$ can yield all the necessary local extensions via specializations. 
To prove Theorem \ref{thm:locdim_not1}, 
we generalize this to finite sets of $G$-extensions $E_i/F_i$, $i=1,\ldots,r$ where $F_i$ is a  
function field in one variable over $K$. 

Since the set of primes splitting completely in finitely many given finite extensions of $K$ is of positive density by Chebotarev's density theorem, it suffices to consider one extension $E_i/F_i$ at a time.
Thus, let $G$ be a finite group possessing a subgroup $A\cong C_p\times C_p$ for some prime $p$, and let $E/F$ be a $G$-extension, with $F$ a finite 
extension of $K(t)$. Consider the Galois closure $\Omega/K(t)$ of $E/K(t)$. This has Galois group embedding into $G\wr H$ where $H$ is the Galois group of the Galois closure of $F/K(t)$. Applying Proposition \ref{prop:locdim_not1} for $\Omega/K(t)$ gives a positive density set $\mathcal{S}$ of primes of $K$ such that  $K_{\fq}$ has an $A$-extension for each $\fq\in \mathcal{S}$, whereas all specializations of $\Omega/K(t)$ at some $t_0\in K_\fq$ have cyclic decomposition group. In particular, for every $K_\fq$-rational place of $F$, the corresponding specialization of $E\cdot K_\fq/F\cdot K_\fq$ has cyclic decomposition group, and a fortiori does not have decomposition group $A$. 
\end{proof}

\begin{rem}
There remains the question which finite groups precisely have local dimension $1$ (over a given number field). Theorem \ref{thm:locdim_not1} reduces the candidate groups to the class of groups in which all Sylow subgroups are cyclic or generalized quaternion (see, e.g., \cite[Chapter XI, Theorem 11.6]{CE}).
In \cite{Koe_dim1}, we will reduce the list much further and in fact achieve a full classification of groups $G$ of local dimension $1$ over (e.g.)  {$\mQ$}. 
{In this case},  $G$ is either (a) cyclic of order $2$ or odd prime power order; or (b) a semidirect product $C\rtimes D$ of two cyclic groups $C,D$ as in (a), with faithful semidirect product action. 
\end{rem}

\begin{rem}
\label{rem:ret}
It would also be interesting to investigate invariants such as $\ld$ or $\hgd$ over fields $K$ other than number fields. Just to give one simple, yet interesting example: If  $K=\mathbb C(x)$ is the fraction field of the Dedekind domain $\mathbb C[x]$, then any finite group $G$ has local dimension $1$ over $K$, as a direct consequence of Riemann's existence theorem: Indeed, every local extension is now of the form $\mathbb{C}(((x-c)^{1/e})) / \mathbb{C}((x-c))$ for some $c\in \mathbb{P}^1(\mathbb{C})$. Choose any $G$-extension $L/\mathbb{C}(t)$ in which a representative of every nontrivial conjugacy class of $G$ occurs as an inertia group generator at some branch point. Then setting $E:=L(x)$, the $G$-extension $E/K(t)$ reaches all local $G$-extensions at all (finite) primes of $K$, simply by specializing $t\mapsto x-c$ with suitable $c\in \mathbb{C}$. On the other hand, \cite{DKLN1} shows that (with very few exceptions) $E/K(t)$ is very far away from reaching all $G$-extensions of $K$ itself by specialization, suggesting again a discrepancy between local and parametric dimension in this case.
\end{rem}


\bibliographystyle{plain}

\end{document}